\documentclass[12pt,reqno]{amsart}
\usepackage{mathtools,calc,verbatim,tikz,url,hyperref,cite,fullpage,caption}
\usepackage{bbm}
\usepackage{amssymb}
\usepackage{amsthm}
\usepackage{amsmath}
\usepackage{graphicx}
\usepackage{latexsym}
\usepackage{color}

\newtheorem{theorem}{Theorem}[section]
\newtheorem{lemma}[theorem]{Lemma}
\newtheorem{cor}[theorem]{Corollary}

\newtheorem{obs}[theorem]{Observation}
\newtheorem{prop}[theorem]{Proposition}

\newtheorem*{claim*}{Claim}

\theoremstyle{definition}
\newtheorem{defn}[theorem]{Definition}

\newtheorem*{qu*}{Question}
\theoremstyle{remark}

\newcommand\N{\mathbb{N}}
\newcommand\R{\mathbb{R}}
\newcommand\Z{\mathbb{Z}}
\newcommand\cA{\mathcal{A}}
\newcommand\cB{\mathcal{B}}

\newcommand\cD{\mathcal{D}}
\newcommand\cE{\mathcal{E}}
\newcommand\cF{\mathcal{F}}
\newcommand\cG{\mathcal{G}}

\newcommand\cI{\mathcal{I}}

\newcommand\cS{\mathcal{S}}
\newcommand\cT{\mathcal{T}}

\renewcommand\Pr{\operatorname{\mathbb{P}}}

\newcommand\eps{\varepsilon}
\renewcommand\leq{\leqslant}
\renewcommand\geq{\geqslant}
\renewcommand\le{\leqslant}
\renewcommand\ge{\geqslant}
\renewcommand\to{\rightarrow}

	\def\eps{\varepsilon}
	\def\Ex{\mathbb{E}}
	
	\def\R{\mathbb{R}}
	\def\Z{\mathbb{Z}}

\def\Bin{\textup{Bin}}
	\def\<{\langle }
	\def\>{\rangle }

\title{The typical structure of sets with small sumset}
\author{Marcelo Campos \and Maur\'icio Collares \and Robert Morris \and \\ Natasha Morrison \and Victor Souza}

\address{Instituto de Ci\^encias Exatas, Universidade Federal de Minas Gerais, Belo Horizonte, MG, Brazil}
\email{mauricio@collares.org}

\address{IMPA, Estrada Dona Castorina 110, Jardim Bot\^anico, Rio de Janeiro, 22460-320, Brazil}
\email{\{rob,souza,marcelo.campos\}@impa.br}

\address{Department of Mathematics and Statistics, University of Victoria, David Turpin Building, 3800 Finnerty Road, Victoria, B.C., Canada V8P 5C2} \email{nmorrison@uvic.ca}

\thanks{The first author was partially supported by CNPq, the second author by PRPq/UFMG (ADRC 11/2017), the third author by CNPq (Proc. 303275/2013-8) and FAPERJ (Proc. 201.598/2014), the fourth author by a CNPq bolsa PDJ, and the fifth author by CAPES}

\begin{document}

\maketitle

\begin{abstract}
In this paper we determine the number and typical structure of sets of integers with bounded doubling. In particular, improving recent results of Green and Morris, and of Mazur, we show that the following holds for every fixed $\lambda > 2$ and every $k \ge (\log n)^4$: if $\omega \to \infty$ as $n \to \infty$ (arbitrarily slowly), then almost all sets $A \subset [n]$ with $|A| = k$ and $|A + A| \le \lambda k$ are contained in an arithmetic progression of length $\lambda k/2 + \omega$.
\end{abstract}

\section{Introduction}

One of the central objects of interest in additive combinatorics is the sumset 
$$A + B := \{ a+b : a \in A, \, b \in B \}$$ 
of two sets $A,B \subset \Z$. A cornerstone of the theory is the celebrated theorem of Fre\u{\i}man~\cite{F66,F73} (later reproved by Ruzsa~\cite{R94}), which states that if $|A+A| \le \lambda |A|$, then $A$ is contained in a generalised arithmetic progression\footnote{That is, a set of the form $P = \big\{ a + i_1 d_1 + \cdots + i_s d_s : i_j \in \{0,\ldots,k_j\} \big\}$ for some $a,d_1,\ldots,d_s,k_1,\ldots,k_s \in \N$.} of dimension $O_\lambda(1)$ and size $O_\lambda(|A|)$, where the implicit constants depend only on $\lambda$. For an overview of the area, see the book by Tao and Vu~\cite{TV}, or the surveys by Green~\cite{G14} and Sanders~\cite{S13}. 

In this paper we will be interested in the number and typical structure of sets with small sumset. This study of this problem was initiated in 2005 by Green~\cite{G05}, who was motivated by applications to random Cayley graphs, and in recent years there has been significant interest in related questions~\cite{ABMS,BLS,BLST,C19,DKLRS,GM}. In particular, Alon, Balogh, Morris and Samotij~\cite{ABMS} conjectured that there are at most
$$2^{o(k)} {\lambda k / 2 \choose k}$$
sets $A \subset \{1,\ldots,n\}$ of size $k$ with $|A+A| \le \lambda k$. This conjecture was proved in the case $\lambda = O(1)$ by Green and Morris~\cite{GM}, and for all $\lambda = o\big( k / (\log n)^3 \big)$ by Campos~\cite{C19}, who moreover (improving a result of Mazur~\cite{Mazur}) showed that almost all such sets are `almost contained' in an arithmetic progression of length $\lambda k / 2 + o(\lambda k)$ (see Theorem~\ref{thm:M:stability}, below). 

Here we will build on this earlier work, and obtain a significantly more precise structural description in the case $\lambda = O(1)$. 
For each $\lambda \ge 3$ and $\eps > 0$, define
\begin{equation}\label{def:c}
c(\lambda,\eps) \,:=\, 2^{20} \lambda^2 \log(1/\eps) \,+\, 2^{560} \lambda^{32}.
\end{equation}
Our main theorem, which determines (up to an additive constant) the length of the smallest arithmetic progression containing a typical set with bounded doubling\footnote{We (informally) call $|A+A|/|A|$ the \emph{doubling} of $A$, so $A$ has bounded doubling if $|A+A| = O(|A|)$.}, 
is as follows. 

\begin{theorem}\label{thm:structure}
Fix $\lambda \ge 3$ and $\eps > 0$, let $n \in \N$ be sufficiently large, and let $k \ge (\log n)^4$. Let $A \subset [n]$ be chosen uniformly at random from the sets with $|A| = k$ and $|A+A| \le \lambda k$. Then there exists an arithmetic progression $P$ with
$$A \subset P \qquad \text{and} \qquad |P| \le \frac{\lambda k}{2} + c(\lambda,\eps)$$ 
with probability at least $1 - \eps$.
\end{theorem}

When $\lambda$ is large and $\eps$ is very small the constant $c(\lambda,\eps)$ is not far from best possible. Indeed, a simple construction (see Section~\ref{lower:sec}) shows that with probability at least $\eps$ the smallest arithmetic progression containing $A$ has size $\lambda k / 2 + \Omega\big( \lambda^2 \log(1/\eps) \big)$.

We will use Theorem~\ref{thm:structure} to deduce the following counting result.

\begin{cor}\label{thm:counting}
For every $\lambda \ge 3$, and every $n,k \in \N$ with $(\log n)^4 \le k = o(n)$, we have
\begin{equation}\label{eq:counting}
\big| \big\{ A \subset [n] : |A| = k, \, |A+A| \leq \lambda k  \big\} \big| = \Theta_\lambda(1) \cdot \frac{n^2}{k} {\lambda k/2 \choose k}.
\end{equation}
\end{cor}

The upper bound in Corollary~\ref{thm:counting} is an almost immediate consequence of Theorem~\ref{thm:structure}, and our lower bound follows from a straightforward calculation (see Sections~\ref{proof:sec} and~\ref{lower:sec}). For both bounds we obtain a constant of the form $\exp\big( \lambda^{\Theta(1)} \big)$ for $\lambda$ large, and it would be interesting to determine the correct exponent of $\lambda$. 

We remark that similar results can be deduced from our proof for all $2 < \lambda < k^{o(1)}$ (see Section~\ref{proof:sec}), but the constant given by our method tends to infinity as $\lambda \to 2$. In order to keep the calculations as simple as possible, we have chosen to focus on the case $\lambda \ge 3$. Let us note here also that the bound $k \ge (\log n)^4$ can be improved somewhat (see Theorem~\ref{thm:structure:quant}); however, some polylogarithmic factor is necessary, since (as was observed in~\cite{C19} and~\cite{GM}) the union of an arithmetic progression of length $k - \lambda + 2$ with $\lambda - 2$ arbitrary points satisfies $|A| = k$ and $|A+A| \le \lambda k$, and there are at least $\Theta(n^{\lambda})$ such sets, which is larger than~\eqref{eq:counting} when $k = o( \log n )$. It seems plausible, however, that Theorem~\ref{thm:structure} and Corollary~\ref{thm:counting} could hold (for $\lambda$ fixed) whenever $k / \log n \to \infty$. 

\pagebreak

In order to understand why Theorem~\ref{thm:structure} should be true, recall first that, by Fre\u{\i}man's theorem, a set has bounded doubling if and only if it is a subset of positive density of a generalised arithmetic progression of bounded dimension. Now, there are $O(n^{d+1})$ generalised arithmetic progressions $P$ of dimension $d$, and if $A$ were a random subset of $P$ of positive density, then $A+A$ would be unlikely to `miss' many elements of $P+P$, which implies that (typically) $|A+A| \ge (d+1+o(1)) \cdot |P|$.\footnote{The factor of $d+1$ is attained by a union of $d$ (unrelated) arithmetic progressions; for most $d$-dimensional generalised arithmetic progressions the doubling would be even larger.} 
This suggests that the number of choices for $A$ should be roughly $n^{d+1} \cdot \binom{\lambda k / (d+1)}{k}$, which (for $k \gg \log n$) is maximised by taking $d = 1$, 
and this leads to the intuition that most sets of bounded doubling should in fact be contained in an arithmetic progression of size roughly $|A+A|/2$. This intuition was partially confirmed in the papers~\cite{GM,Mazur,C19} mentioned above, which showed that there typically exists an arithmetic progression $P$ of length $(1/2 + o(1)) |A+A|$ such that $|A \setminus P| = o(|A|)$. In fact, it was shown in~\cite{C19} that this holds even when $|A+A| / |A|$ is much larger, see Theorem~\ref{thm:M:stability}, below. 


The main tool in the proof of Theorem~\ref{thm:structure} is a `container theorem' for sets with small doubling (see Theorem~\ref{thm:M:containers}, below), which was proved by Campos~\cite{C19} using the so-called method of hypergraph containers (see, e.g.,~\cite{BMS18}). More precisely, Campos used the \emph{asymmetric} container lemma of Morris, Samotij and Saxton~\cite{MSS}, which is a variant of the original container lemma of Balogh, Morris and Samotij~\cite{BMS} and Saxton and Thomason~\cite{ST}, to resolve the conjecture of Alon, Balogh, Morris and Samotij~\cite{ABMS} mentioned above, and to determine the (rough) typical structure of a set with given doubling. 

We will use this container theorem in three different ways: first, to control the rough structure of a set with bounded doubling (see Theorem~\ref{thm:M:stability} and Lemma~\ref{lem:stability:first}); then to prove a variant of a probabilistic lemma of Green and Morris~\cite{GM} (see Lemma~\ref{lem:prob_sumset}); and finally to control the fine structure of the set near the ends of the progression containing it (see Section~\ref{sec:dense}). We consider this last step to be the most interesting aspect of the proof, since we are not aware of any previous application of containers to the task of `cleaning up' a set, that is, replacing a rough structural result with a precise one. We hope that our proof will inspire further applications of this type in other combinatorial settings.

\section{An overview of the proof}\label{overview:sec}

In this section we will prepare the reader for the details of the proof by giving a rough outline of the main ideas. Let us fix $\lambda \ge 3$, and let $k \in \N$ be sufficiently large. We will mostly work with sets of integers that are `close' to being a subset of the interval $[\lambda k / 2]$, since the stability theorem proved in~\cite{C19} (see Theorem~\ref{thm:M:stability}, below) implies that almost all of the sets that we need to count are close to an arithmetic progression of length $\lambda k / 2$, and any such progression can be mapped into $[\lambda k / 2]$ (see Section~\ref{stability:sec} for the details).\footnote{For simplicity, we will assume throughout the paper that $\lambda k / 2$ is an integer.} 

Given a set $A \subset \Z$, let us write
\begin{equation}\label{def:b:and:r}
b(A) := |A \setminus [\lambda k / 2]| \qquad \text{and} \qquad r(A) := \max(A) - \min(A) - \lambda k / 2.
\end{equation}
Let us also fix $\eps > 0$ and set $\delta := 2^{-32} \lambda^{-3}$. By Lemma~\ref{lem:stability}, below, the problem will reduce to bounding the size of the following family of sets. 

\begin{defn}\label{def:I}
Let $\cI$ denote the family of sets $A \subset \{ -\lambda k/2,\ldots,\lambda k\}$ with $|A| = k$ and $|A+A| \le \lambda k$, such that  
$$b(A) \le \delta k \qquad \qquad \text{and} \qquad \qquad r(A) \ge c(\lambda,\eps),$$
and the sets $\big\{ x \in A : x \le 0 \big\}$ and $\big\{ x \in A : x > \lambda k / 2 \big\}$ are non-empty.
\end{defn}

We will partition the family $\cI$ according to the `density' of the set $B := A \setminus [\lambda k / 2]$. To be precise, set 
\begin{equation}\label{def:f}
f(\lambda) := 2^{10} \lambda^3,
\end{equation}
and say that $B$ is \emph{sparse} if $r(A) > f(\lambda) b(A)$. The following lemma, which is proved in Section~\ref{sec:spreadout}, bounds the number of sets $A \in \cI$ such that $B$ is sparse. 

\begin{lemma}\label{lem:sparse}
For every $\lambda \ge 3$ and $\eps \in (0,1)$, and every $k \in \N$, we have
$$\Big| \Big\{ A \in \cI \,:\, r(A) > f(\lambda) b(A) \Big\} \Big| \le \frac{\eps}{\lambda^3} {\lambda k/2 \choose k}.$$
\end{lemma}

To bound the number of choices for $A$, we will bound separately the choices for $B$ and $A' := A \setminus B$. 
Assume (for simplicity) that $\min(A) = 0$, so $\max(A) = \lambda k / 2 + r$. The proof of Lemma~\ref{lem:sparse} uses the following simple idea: the set $(A' + \max(A)) \setminus [\lambda k]$ typically contains about $2r/\lambda$ elements, and this restricts the size of the set $A'+ A'$, and hence the number of choices for $A'$. More precisely, we will use a straightforward counting argument when $(A' + \max(A)) \setminus [\lambda k]$ is much smaller than $r/\lambda$ (see Lemma~\ref{lem:claim2}), and an application of the container theorem 
when it is larger (see Lemma~\ref{lem:claim1}). Moreover, the assumption that $B$ is sparse allows us to (trivially) bound the number of choices for $B$.

We remark that our application of the container theorem in the proof of Lemma~\ref{lem:sparse} proceeds via a probabilistic lemma (Lemma~\ref{lem:prob_sumset}), which is a generalisation of a result of Green and Morris~\cite{GM}. This lemma gives a (close to tight) upper bound on the number of $k$-subsets of $[n]$ whose sumset missed many elements of $\{2,\ldots,2n\}$, and is proved in Section~\ref{GM:lemma:sec}, using the container theorem of Campos~\cite{C19} mentioned in the introduction. 
 
 \medskip
 \pagebreak

When $r(A) \le f(\lambda) b(A)$, we will say that the set is \emph{dense}. In Sections~\ref{sec:somewhat:dense} and~\ref{sec:dense} we will prove the following lemma, which bounds the number of dense sets in $\cI$. 

\begin{lemma}\label{lem:dense}
For every $\lambda \ge 3$ and $\eps \in (0,1)$, and every $k \in \N$, we have
$$\Big| \Big\{ A \in \cI \,:\, r(A) \le f(\lambda) b(A) \Big\} \Big| \le \frac{\eps}{\lambda^3} {\lambda k/2 \choose k}.$$
\end{lemma}

The proof of Lemma~\ref{lem:dense} is significantly more difficult than that of Lemma~\ref{lem:sparse}, and is the most interesting and novel part of the argument, involving a surprising and unusual application of the container method. Set $A' := A \cap [\lambda k / 2]$ and $B := A \setminus A'$, as before, and suppose that $|B| = b$ and $|(B+B) \setminus [\lambda k]| = \mu b$. The main difficulties arise when $r = O(\mu b)$ and $\mu = \Theta(\lambda)$, and we first take care of the remaining cases in Section~\ref{sec:somewhat:dense}. 

For these `easy' cases (see Lemmas~\ref{lem:r:vs:b}  and~\ref{lem:big:or:small:sumset}) we use similar ideas to those used in the proof of Lemma~\ref{lem:sparse} (see the sketch above), except that instead of using a trivial bound, we will need to apply the container theorem (via Theorem~\ref{thm:M:counting}) in order to bound the number of choices for the set $B$ (see Lemma~\ref{lem:count:Bs}), and the calculations are significantly more delicate. In particular, we will need to use our bounds on the size of both $(A' + \max(A)) \setminus [\lambda k]$ 
and $(B+B) \setminus [\lambda k]$ to bound the size of $A' + A'$, and thus the number of choices for $A'$. 

Counting the sets with $r = O(\mu b)$ and $\mu = \Theta(\lambda)$ is the most interesting part of the proof. The key idea is to use the container theorem to obtain a collection of `containers' $(C,D)$ for the `missing' set $M(A) := [\lambda k] \setminus (A+A)$, which is typically (see Lemma~\ref{lem:hitmiddle}) contained in the set $Y+Y$, where $Y$ is the set of points that are `close' to the endpoints of $[\lambda k/2]$. The containers satisfy $M(A) \subset C$ and $A \cap Y \subset D$, and moreover $D$ misses roughly $|C|/2$ points of $Y$ (for the precise statement, see Corollary~\ref{cor:M:containers:app}). The key step (Lemma~\ref{lem:finalcount}) then uses these properties to bound the number of sets $A$ corresponding to each container. 
Taking a union bound over containers, it follows that there are at most
$$\exp\bigg( - \frac{r}{2^{19} \lambda^2} \bigg) \binom{\lambda k/2}{k}$$ 
sets $A \in \cI$ with $r(A) = r \le f(\lambda) b(A)$, and this easily implies Lemma~\ref{lem:dense}. 

The rest of the paper is organised as follows. First, in Section~\ref{containers:sec}, we recall the main results of~\cite{C19}, and deduce the container theorem we will use in the proof (Corollary~\ref{cor:M:containers}). In Section~\ref{GM:lemma:sec} we use this container theorem to prove the probabilistic lemma mentioned above (Lemma~\ref{lem:prob_sumset}), and in Section~\ref{stability:sec} we will use the results of~\cite{C19} to reduce the problem to that of bounding the size of the set $\cI$. In Section~\ref{sec:spreadout} we prove Lemma~\ref{lem:sparse}, in Sections~\ref{sec:somewhat:dense} and~\ref{sec:dense} we prove Lemma~\ref{lem:dense}, and in Section~\ref{proof:sec} we put the pieces together and prove Theorem~\ref{thm:structure}. Finally, in Section~\ref{lower:sec}, we provide two simple constructions that show that the upper bounds in Theorem~\ref{thm:structure} and Corollary~\ref{thm:counting} are not far from best possible.

\medskip

\section{The container theorem}\label{containers:sec}

In this section we will recall the main results of~\cite{C19}, which will play an important role in the proofs of our main theorems. We begin by stating 
the main container theorem from~\cite{C19}. 

\begin{theorem}[Theorem~4.2 of~\cite{C19}]\label{thm:M:containers}
Let $m \ge (\log n)^2$, let $Y \subset \Z$ with $|Y| = n$, and let $0 < \gamma < 1/4$. There is a family $\cA \subset 2^{Y+Y} \times 2^Y$ of pairs of sets $(A,B)$, of size
\begin{equation}\label{eq:number:of:original:containers}
|\cA| \leq \exp\Big( 2^{16} \gamma^{-2} \sqrt{m} \, (\log n)^{3/2} \Big),
\end{equation}
such that:
\begin{itemize}
\item[$(i)$] For each $J \subset Y$ with $|J + J| \le m$, there is $(A,B) \in \cA$ with $A \subset J+J$ and $J \subset B$.\smallskip
\item[$(ii)$] For every $(A,B) \in \cA$, $|A| \leq m$ and either $|B| \leq \frac{m}{\log n}$ or there are at most $\gamma^2 |B|^2$ pairs $(b_1, b_2) \in B \times B$ such that $b_1 + b_2 \notin A$.
\end{itemize}
\end{theorem}

The reader may find it useful to imagine the statement of Theorem~\ref{thm:M:containers} as saying that for each set $J \subset Y$, there exists a `container' $(A,B) \in \cA$ such that 
$$J \subset B, \qquad B + B \approx A \qquad \text{and} \qquad A \subset J+J.$$ 
Moreover, and crucially, the number of containers is sub-exponential in $m$. 
 

We will also use the following two consequences of Theorem~\ref{thm:M:containers}, which were both proved in~\cite{C19}. The first determines the number of sets $A \subset [n]$ with $|A| = k$ and $|A+A| \le \lambda k$ up to a factor of $2^{o(k)}$. We will use it in Section~\ref{sec:somewhat:dense} to bound the number of choices for $A \setminus [\lambda k / 2]$. 

\begin{theorem}[Theorem~4.1 of~\cite{C19}]\label{thm:M:counting}
Let $n,k \in \N$, and let $2 < \lambda < 2^{-36}\frac{k}{(\log n)^3}$. The number of sets $A \subset [n]$ with $|A| = k$ such that $|A+A| \le \lambda k$ is at most 
$$\exp\Big(2^9 \lambda^{1/6} k^{5/6} \log k \sqrt{\log n}\Big){\lambda k / 2 \choose k}.$$
\end{theorem}

The second gives structural information about  
a typical set with small doubling; we will use it in Section~\ref{stability:sec}. The following is a slight generalisation of~\cite[Theorem~5.1]{C19}, but follows from the same proof; for completeness, the details are given in Appendix~\ref{app:stability}. 

\begin{theorem}[Theorem~5.1 of~\cite{C19}]\label{thm:M:stability}
Let $n,k \in \N$ and $2 \leq \lambda \leq 2^{-120} \frac{k}{(\log n)^3}$, and let $2^{8} \lambda^{1/6} k^{-1/6} \sqrt{\log n} \le \gamma < 2^{-8}$. For all but at most
$$e^{- \gamma k} {\lambda k / 2 \choose k}$$
sets $A \subset [n]$ with $|A| = k$ and $|A+A| \le \lambda k$, the following holds: there exists $T \subset A$, with $|T| \le 2^{9} \gamma k$, such that $A \setminus T$ is contained in an arithmetic progression of size $\lambda k / 2 + 2^{7} \gamma \lambda k$.
\end{theorem}

\pagebreak

The upper bounds on $\lambda$ in Theorems~\ref{thm:M:counting} and~\ref{thm:M:stability} are the reason why we require the bound $k \ge (\log n)^4$ in Theorem~\ref{thm:structure} and Corollary~\ref{thm:counting}. We remark that some $\log$-factor is necessary here, since it was observed in~\cite{C19} that the conclusions of the theorems fail to hold if $k = o\big( \lambda \log n \big)$, cf.~the discussion in the introduction. 

We will apply Theorem~\ref{thm:M:containers} (in Sections~\ref{GM:lemma:sec} and~\ref{sec:dense}) via the following corollary.

\begin{cor}\label{cor:M:containers}
Let $0 < \gamma < 1/4$, let $S_1, S_2 \subset \Z$ be intervals, and set 
\begin{equation}\label{def:XandY}
Y := S_1 \cup S_2 \qquad \text{and} \qquad X := (S_1 + S_1) \cup (S_2 + S_2).
\end{equation}
Then there is a family $\cB \subset 2^{X} \times 2^{Y}$ of size at most 
\begin{equation}\label{eq:number:of:containers}
\exp\Big( 2^{17} \gamma^{-2}\sqrt{|Y|} \, \big( \log |Y| \big)^{3/2} \Big)
\end{equation}
such that:
\begin{itemize}
\item[$(a)$] For every pair of sets $U \subset Y$ and $W \subset X \setminus (U + U)$, there exists $(C, D) \in \cB$ such that $W \subset C$ and $U \subset D$.
\item[$(b)$] For every $(C,D) \in \cB$, 
\begin{equation}\label{eq:size:of:D}
|D| \leq \max\bigg\{ (1 + 4\gamma) |Y| - \frac{|C|}{2}, \, \frac{3|Y|}{\log |Y|} \bigg\}.
\end{equation}
\end{itemize}
\end{cor}

Note that replacing $W \subset X \setminus (U + U)$ by $W = X \setminus (U + U)$ in part $(a)$ would give an equivalent statement; however, we will find this formulation convenient. To deduce Corollary~\ref{cor:M:containers} from Theorem~\ref{thm:M:containers}, we will need the following easy lemma, cf.~\cite[Corollary~3.3]{C19}.

\begin{lemma}\label{prop:trivsupersat}
Let $\gamma > 0$, let $S_1,S_2 \subset \Z$ be intervals, and set 
$$Y := S_1 \cup S_2 \qquad \text{and} \qquad X := (S_1 + S_1) \cup (S_2 + S_2).$$ 
Let $C \subset X$ and $D \subset Y$. If 
$$|D| \ge (1+4\gamma)|Y| - |C|/2$$
then there are at least $\gamma^2 |D|^2$ pairs $(b_1,b_2) \in D \times D$ such that $b_1+b_2 \in C$. 
\end{lemma}

\begin{proof}
Suppose first that $S_1 \cap S_2$ is non-empty, so $X = Y + Y$, and let the elements of $D$ be $d_1 < \cdots < d_\ell$. Then $D + D \subset X$ contains the $2\ell - 1$ elements
$$d_1 + d_1 < d_1 + d_2 < \cdots < d_1 + d_\ell < d_2 + d_\ell < \cdots < d_\ell + d_\ell,$$
and $2\ell - 1 \ge (2 + 8\gamma)|Y| - |C| - 1 = |X| - |C| + 8\gamma |Y|$, since $|X| = 2|Y| - 1$. Since $C \subset X$, it follows that there are at least $8\gamma |Y|$ pairs $(b_1,b_2) \in D \times D$ such that $b_1+b_2 \in C$ and $\{b_1,b_2\} \cap \{d_1,d_\ell\}$ is non-empty. Removing $d_1$ and $d_\ell$ from $D$, and repeating the argument $\gamma |Y|$ times, we obtain $\gamma^2 |Y|^2$ pairs $(b_1,b_2) \in D \times D$ such that $b_1+b_2 \in C$.

When $S_1$ and $S_2$ are disjoint, we simply apply the argument above for the two sets $D_1 := D \cap S_1$ and $D_2 := D \cap S_2$. To spell out the details, for each $i \in \{1,2\}$ there are $2|D_i| - 1$ pairs $(b_1,b_2) \in D_i \times D_i$ with distinct sums such that either $b_1 = \min(D_i)$ or $b_2 = \max(D_i)$. Moreover, $D_1 + D_1$ and $D_2+D_2$ are disjoint subsets of $X$, and
$$2|D| - 2 \ge (2 + 8\gamma)|Y| - |C| - 2 = |X| - |C| + 8\gamma |Y|,$$
since $|X| = 2|Y| - 2$. As before, it follows that there are at least $8\gamma |Y|$ pairs $(b_1,b_2) \in D \times D$ such that $b_1+b_2 \in C$ and either $b_1 \in \{\min(D_1),\, \min(D_2) \}$ or $b_2 \in \{\max(D_1), \,\max(D_2) \}$. Removing the minimum and maximum elements of $D_1$ and $D_2$, and repeating the argument $\gamma |Y|$ times, we obtain $\gamma^2 |Y|^2$ pairs $(b_1,b_2) \in D \times D$ such that $b_1+b_2 \in C$, as claimed.
\end{proof}

\begin{proof}[Proof of Corollary~\ref{cor:M:containers}]
Applying Theorem~\ref{thm:M:containers} with $n := |Y|$ and $m := 3|Y|$, we obtain a family $\cA \subset 2^{Y+Y} \times 2^Y$, with 
$$|\cA|\leq \exp\Big( 2^{17}\gamma^{-2}\sqrt{|Y|} \big( \log |Y| \big)^{3/2} \Big),$$ 
satisfying properties~$(i)$ and~$(ii)$ of the theorem. We claim that 
$$\cB:=\big\{ (X \setminus A,B) : (A,B)\in \cA \big\} \subset 2^{X} \times 2^{Y}$$
satisfies properties~$(a)$ and~$(b)$ of Corollary~\ref{cor:M:containers}.

To show that property~$(a)$ holds, let $U \subset Y$ and $W \subset X \setminus (U + U)$, and set $J := U$. 
Noting that $J \subset Y$, 
and that 
$$|J+J| \le |Y+Y| \le 3|Y| = m,$$ 
it follows from Theorem~\ref{thm:M:containers}$(i)$ that there exists $(A,B) \in \cA$ with $A \subset J+J$ and $J \subset B$, and hence there exists $(C, D) = (X \setminus A,B) \in \cB$ such that $W \subset C$ and $U \subset D$.

For property~$(b)$, let $(C,D) \in \cB$, and observe that, by Theorem~\ref{thm:M:containers}$(ii)$, either $|D| \leq \frac{3|Y|}{\log |Y|}$, or there are at most $\gamma^2 |D|^2$ pairs $(b_1, b_2) \in D \times D$ such that $b_1 + b_2 \in C$. In the latter case, we have $|D| \le (1 + 4\gamma)|Y| - |C|/2$, by Lemma~\ref{prop:trivsupersat}. Since $|\cB| \le |\cA|$, the corollary follows. 
\end{proof}

\section{A probabilistic lemma}\label{GM:lemma:sec}

Green and Morris~\cite[Theorem~1.3]{GM} used their bounds on the number of sets with small sumset to prove that if $S$ is a random subset of $\N$, with each element included in $S$ independently with probability $1/2$, then 
$$\Pr\Big( \big| \N \setminus \big( S + S \big) \big| \ge m \Big) = 2^{-m/2 + o(m)}.$$ 
We will use Corollary~\ref{cor:M:containers} to prove the following generalisation of their theorem. We remark that a similar result (with a slightly larger error term) for larger values of $k$ 
can be deduced from exactly the same proof. 

\pagebreak

\begin{lemma}\label{lem:prob_sumset}
Let $n,k \in \N$, with $k \le 2n/3$, 
and set $p := k/n$. 
If $S$ is a uniformly-chosen random subset of $[n]$ of size $k$, then
\begin{equation}\label{eq:prob}
\Pr\Big( \big| \big\{ 2,\ldots,2n \big\} \setminus \big( S + S \big) \big| \ge m \Big) 
\le  \exp\big( 2^{16} m^{7/8} \big) \cdot \big( 1 - p \big)^{m/2}.
\end{equation}
\end{lemma}

In the proof of Lemma~\ref{lem:prob_sumset} we will also use the following well-known inequality (see, e.g.,~\cite[Lemma 5.2]{ABMS1}). 

\begin{lemma}[Pittel's inequality]\label{lem:pittel}
Let $n,k \in \N$ with $k \le n$, and set $p := k/n$. If $\cI$ is a monotone decreasing property on $[n]$, then
$$\Pr\big( \cI\text{ holds for a random $k$-subset of }[n] \big) \le 2 \cdot \Pr\big( \cI\text{ holds for a $p$-random subset of }[n] \big).$$
\end{lemma}

\begin{proof}
Following the proof in~\cite{ABMS1}, recall that $\Bin(n,p) \le \lceil pn \rceil = k$ holds with probability at least $1/2$. Since $\cI$ is monotone decreasing, the claimed bound follows.
\end{proof}

We first prove a simple lemma that will also be useful in Section~\ref{sec:dense}.

\begin{lemma}\label{lem:middle:covered}
Let $n \in \N$ and $k \in [n]$, set $p := k/n$, and let $M \in \N$. If $S$ is a uniformly-chosen random subset of $[n]$ of size $k$, then
$$\Pr\Big( \big\{ M+1, \dotsc, 2n - M + 1 \big\} \not\subset S + S \Big) \le \frac{8}{p^2} \cdot \big(1 - p^2\big)^{M/2}.$$
\end{lemma}

\begin{proof}
Observe that the left-hand side is at most
$$\sum_{x = M + 1}^{2n - M + 1} \Pr\big( x \notin S + S \big) \le 2 \sum_{x = M + 1}^{n + 1} \Pr\big( x \notin S + S \big),$$
since, by symmetry, $\Pr\big( x \notin S + S \big) = \Pr\big( 2n +2 - x \notin S + S \big)$. Now, for $x \leq n + 1$, we can use Pittel's inequality to bound
$$\Pr\big( x \notin S + S \big) = \Pr\bigg( \bigcap_{i = 1}^{\lfloor x/2 \rfloor} \Big( \big\{ i \notin S \big\} \cup \big\{ x - i \notin S \big\} \Big) \bigg) \le 2 \big( 1- p^2 \big)^{(x-1)/2}.$$
It follows that
$$\Pr\Big( \big\{ M+1, \dotsc, 2n - M + 1 \big\} \not\subset S + S \Big) \le 4 \sum_{x = M + 1}^{\infty} \big( 1- p^2 \big)^{(x-1)/2} \le \frac{8}{p^2} \big(1 - p^2\big)^{M/2},$$
as claimed. 
\end{proof}

We are now ready to deduce Lemma~\ref{lem:prob_sumset} from Corollary~\ref{cor:M:containers}.

\begin{proof}[Proof of Lemma~\ref{lem:prob_sumset}] 
Observe first that, since $1 - p \ge e^{-2p}$ for $0 \le p \le 2/3$, the claimed bound holds trivially if $pm \le 2^{16} m^{7/8}$. We may therefore assume that $m \ge 2^{128} p^{-8}$. 

We will use Lemma~\ref{lem:middle:covered} to deal with the case that the `middle' is not covered by $S+S$. To be precise, set $M := \lceil 4m/p \rceil$ and let us write $\cE$ for the event that $\big\{ 2M+1, \dotsc, 2n - 2M + 1 \big\} \subset S + S$. Note that if $\cE$ holds, then 
$$\{2, \dotsc, 2n\} \setminus (S + S) \subset X := \big\{ 2, \dotsc, 2M \big\} \cup \big\{ 2n - 2M + 2, \dotsc, 2n \big\}.$$ 

\noindent Setting $W := X \setminus ( S + S )$, it follows that
$$\Pr\Big( \big| \big\{ 2,\ldots,2n \big\} \setminus \big( S + S \big) \big| \ge m \Big) \leq \Pr\big( |W| \geq m \big) + \Pr(\cE^c).$$
By Lemma~\ref{lem:middle:covered}, we have\footnote{Note that if $m \ge pn/4$, then $M \ge n$, and so the event $\cE$ holds trivially.} 
$$\Pr( \cE^c ) \le \frac{8}{p^2} \big(1 - p^2\big)^{M} \le \frac{8}{p^2} \big( 1 - p \big)^{m},$$
where the second inequality follows since $1 - x^2 \le (1-x)^{x/2}$ for all $0 \le x \le 1$.


To complete the proof, we will use Corollary~\ref{cor:M:containers} to bound the probability that $|W| \ge m$. Indeed, applying the corollary to the set
$$Y := \big\{ 1, \dotsc, M \big\} \cup \big\{ n - M + 1, \dotsc, n \big\},$$
and noting that the set $X$ defined above is the same as that defined in~\eqref{def:XandY}, we obtain a family $\cB \subset  2^{X} \times 2^{Y}$ of containers of size at most
\begin{equation}\label{eq:bound:on:B}
\exp\Big( 2^{18} \gamma^{-2}\sqrt{M} (\log M)^{3/2} \Big) = \big( 1 - p \big)^{-\gamma M},
\end{equation}
where $\gamma > 0$ is chosen so that the equality holds. In particular, if $(C,D) \in \cB$, then
\begin{equation}\label{eq:size:of:D:repeat}
|D| \leq \max\bigg\{ (1 + 4\gamma) |Y| - \frac{|C|}{2}, \, \frac{3|Y|}{\log |Y|} \bigg\},
\end{equation}
and if $U \subset Y$ and $W \subset X \setminus (U + U)$, then there exists $(C, D) \in \cB$ with $W \subset C$ and $U \subset D$. 

To apply Corollary~\ref{cor:M:containers}, we need to check that $\gamma < 1/4$. Using the bounds $1-p \leq e^{-p}$ and $M \ge m/p$, and noting that the function $x \mapsto (\log x)^{3/2} / \sqrt{x}$ is decreasing for $x > 2^{5}$, it follows from~\eqref{eq:bound:on:B} that
$$\gamma^3 \le \frac{2^{18} (\log M)^{3/2}}{p\sqrt{M}} \leq \frac{2^{18}}{\sqrt{pm}} \left(\log \frac{m}{p}\right)^{3/2}.$$
Therefore, since $M \le 8m/p$, we have
\begin{equation}\label{eq:epsM:vs:m}
\gamma M \le \frac{8\gamma m}{p} \leq \frac{2^9 m^{5/6}}{p^{7/6}}\left(\log \frac{m}{p}\right)^{1/2} < m,
\end{equation}
where the final inequality follows from the assumption that $m \ge 2^{128} p^{-8}$. Since $M \ge 4m$, it follows from~\eqref{eq:epsM:vs:m} that $\gamma < 1/4$, and so this is a valid choice of $\gamma$ in Corollary~\ref{cor:M:containers}.

We next claim that
\begin{equation}\label{eq:sum:over:containers}
\Pr\big( |W| \ge m \big) \le \sum_{(C,D) \in \cB} \Pr \Big( \big( W \subset C \big) \cap \big( S \cap Y \subset D \big) \Big).
\end{equation}
To see this, observe first that 
$$W = X \setminus ( S + S ) \subset X \setminus \big( (S \cap Y) +(S \cap Y) \big)$$
since $S \cap Y \subset S$. By the property of $\cB$ guaranteed by Corollary~\ref{cor:M:containers}$(a)$, applied with $U :=  S \cap Y$, 
it follows that there exists a pair $(C, D) \in \cB$ with $W \subset C$ and $S \cap Y \subset D$.


To bound the right-hand side of~\eqref{eq:sum:over:containers}, observe first that
\begin{equation}\label{eq:SYD}
\Pr \big( S \cap Y \subset D \big) \le \binom{n - |Y \setminus D|}{pn} \binom{n}{pn}^{-1}
\end{equation}
for every $(C,D) \in \cB$, since $S$ is a uniformly-chosen set of size $k = pn$, and if $S \cap Y \subset D$ then $S \cap (Y \setminus D) = \emptyset$. Moreover, by~\eqref{eq:size:of:D:repeat}, 
if $|W| \ge m$ then
\begin{equation}\label{eq:Dbound}
|Y \setminus D| \ge |Y| - |D| \ge \frac{m}{2} - 8\gamma M
\end{equation}
for every $(C,D) \in \cB$ with $W \subset C$. It follows from~\eqref{eq:bound:on:B},~\eqref{eq:sum:over:containers},~\eqref{eq:SYD} and~\eqref{eq:Dbound} that
\begin{equation}\label{eq:Wbound:penultimate}
\Pr\big( |W| \ge m \big) \, \le \, \big( 1 - p \big)^{-\gamma M} \binom{n - m/2 + 8 \gamma M}{pn} \binom{n}{pn}^{-1} \le \, \big( 1 - p \big)^{m/2 - 9\gamma M},
\end{equation}
where the second inequality follows from the standard binomial inequality
\begin{equation}\label{fact:binomial_classic}
\binom{a-c}{b} \le \bigg( \frac{a - b}{a} \bigg)^c \binom{a}{b}.
\end{equation}
Combining~\eqref{eq:epsM:vs:m} and~\eqref{eq:Wbound:penultimate}, and noting that $1 - p \ge e^{-2p}$ for $0 \le p \le 2/3$, it follows that
$$\Pr\big( |W| \ge m \big) \, \le \, \exp\Big( 2^{14} m^{5/6} p^{-1/6}(\log m)^{1/2} \Big) \cdot \big( 1 - p \big)^{m/2}.$$
Since $p^{-1/6}( \log m )^{1/2} \le m^{1/24}$, by our lower bound on $m$, the claimed bound follows.
\end{proof}

We will usually apply Lemma~\ref{lem:prob_sumset} in the following form. Recall that $\delta = 2^{-32} \lambda^{-3}$.

\begin{cor}\label{cor:prob:sumset}
Let $\lambda \ge 3$ and $k,m,b \in \N$, with $m \ge 2^{400} \lambda^{24}$ and $b \le \delta k$. There are at most 
$$e^{2 \delta m} \left(\frac{\lambda - 2}{\lambda}\right)^{m/2} {\lambda k / 2 \choose k - b}$$
sets $A' \subset [\lambda k / 2]$ of size $k - b$ such that $\big| [\lambda k] \setminus ( A' + A' ) \big| \ge m$. 
\end{cor}

\begin{proof}
We simply apply Lemma~\ref{lem:prob_sumset} with $p = 2(k-b) / \lambda k \le 2/3$, and observe that
$$\exp\big( 2^{16} m^{7/8} \big) \big( 1 - p \big)^{m/2} \le e^{2 \delta m}\left(\frac{\lambda - 2}{\lambda}\right)^{m/2},$$
by our bounds on $b$ and $m$. To spell out the details, note that 
$$2^{16} m^{7/8} \le \delta m,$$
since $\delta = 2^{-32} \lambda^{-3}$ and $m \ge 2^{400} \lambda^{24}$. 
Now, observe that 
$$\big( 1 - p \big)^{m/2} \le \bigg( \frac{\lambda - 2 + 2\delta}{\lambda}  \bigg)^{m/2} \le \exp\bigg( \frac
{\delta m}{\lambda - 2} \bigg) \bigg( \frac{\lambda  - 2}{\lambda} \bigg)^{m/2}.$$
Since $\lambda \ge 3$, the claimed bound follows.
\end{proof}

Since we will often only need a weaker bound, let us note here, for convenience, that 
\begin{equation}\label{eq:prob:sumset:weak}
e^{2 \delta m} \left(\frac{\lambda - 2}{\lambda}\right)^{m/2} \le \left(\frac{\lambda - 1}{\lambda}\right)^{m/2},
\end{equation}
since $\delta = 2^{-32} \lambda^{-3}$. 

\subsection{Tools and inequalities}

To finish this section, let us state some standard tools that we will use in the proof of Theorem~\ref{thm:structure}. The first is known as Ruzsa's covering lemma (see, e.g.,~\cite[Lemma~2.14]{TV}), and was first proved in~\cite{R99}. For completeness, we give the proof. 

\begin{lemma}[Ruzsa's covering lemma]\label{lem:RCL}
Let $A,B \subset \Z$ be non-empty sets of integers, and suppose that $|A+B| \le \mu |A|$. Then there exists a set $X \subset B$ with $|X| \le \mu$ such that
$$B \subset A - A + X.$$
\end{lemma} 

\begin{proof}
Let $X \subset B$ be maximal such that the sets $A + x$ for $x \in X$ are disjoint. Observe that $|A+B| \ge |A| |X|$, and therefore $|X| \le \mu$. Now, since $X$ is maximal, $A + b$ intersects $A + X$ for every $b \in B \setminus X$, and hence $B \subset A - A + X$, as claimed. 
\end{proof}

We will also use the following special case of the Pl\"unnecke--Ruzsa inequalities~\cite{P69,P70,R89},  which is also an immediate consequence of Ruzsa's triangle inequality~\cite{R78}.

\begin{lemma}[Pl\"unnecke--Ruzsa inequality]\label{lem:PR}
If $|A+A| \le \lambda |A|$, then $|A - A| \le \lambda^2 |A|$. 
\end{lemma}

\begin{proof}
To prove that $|A - A| \cdot |A| \le |A+A|^2$, it suffices to construct an injective map $\varphi \colon (A - A) \times A \to (A+A)^2$. To do so, choose an arbitrary function $f \colon A - A \to A^2$ such that if $f(x) = (a,b)$ then $a - b = x$, and define $\varphi(x,c) \mapsto (a+c,b+c)$, where $(a,b) = f(x)$. To see that $\varphi$ is injective, observe that $x = (a+c) - (b+c)$ and that $(a,b) = f(x)$. 
\end{proof}

In Section~\ref{sec:somewhat:dense} we will use a simple special case of the following result of Fre\u{\i}man~\cite{F59}.

\begin{lemma}[Fre\u{\i}man's $3k-4$ theorem]\label{lem:F3k4}
If $|A + A| \le 3|A| - 4$, then $A \subset P$ for some arithmetic progression $P$ of size $|A+A| - |A| + 1$.
\end{lemma}

We will also make frequent use of the following standard inequality in the calculations below:
\begin{equation}\label{eq:binomial:classic:too}
\binom{a-c}{b-d} \le \left( \frac{a - c}{a}\right)^{b - d} \left( \frac{b}{a-b}\right)^d\binom{a}{b}.
\end{equation}
In particular, note that
 \begin{equation}\label{eq:binomial:usual:application}
{\lambda k / 2 \choose k - b} \le \left( \frac{2}{\lambda - 2}\right)^b {\lambda k / 2 \choose k}.
\end{equation}

\smallskip
\noindent We will also use the following inequality once, in Section~\ref{sec:somewhat:dense}.

\begin{obs}\label{obs:binomial}
$$\binom{ca}{a} \leq \bigg( \frac{c^c}{(c-1)^{c-1}} \bigg)^a,$$
for every $a \in \N$ and $1 < c \in \R$. 
\end{obs} 

\begin{proof}
Set $y = (c-1)^{1/c}$, and note that $y/(c-1) = y^{1-c}$. It follows that 
\begin{align*}
\bigg( \frac{c^c}{(c-1)^{c-1}} \bigg)^a & \, = \,  \bigg( \bigg( 1+\frac{1}{c-1} \bigg) (c - 1)^{1/c}\bigg)^{ca} \\
& \, = \, \big(y + y^{1-c}\big)^{ca} = \sum_{i=0}^{ca} \binom{ca}{i} y^{ca - i} \cdot y^{(1-c)i} \ge \binom{ca}{a},
\end{align*}
where the last step follows by considering the term $i = a$.
\end{proof}


\section{Reducing to an interval}\label{stability:sec}

Let us fix $\lambda \ge 3$, and for each $n,k \in \N$ define
\begin{equation}\label{def:Lambda}
\Lambda = \Lambda(n,k) := \big\{ A \subset [n] \,:\, |A| = k \, \text{ and } \, |A+A| \le \lambda k \big\}.
\end{equation}
Let us also fix $\eps \in (0,1)$ (since Theorem~\ref{thm:structure} holds trivially for $\eps \ge 1$) and, writing $\ell(A)$ for the length of the smallest arithmetic progression containing $A$, define
\begin{equation}\label{def:Lambda:star}
\Lambda^* = \Lambda^*(n,k) := \big\{ A \in \Lambda : \ell(A) \le \lambda k/2 + c(\lambda,\eps) \big\}.
\end{equation}
In this section we will prove the following lemma, which reduces the problem of bounding $|\Lambda \setminus \Lambda^*|$ to that of bounding $|\cI|$ (see Definition~\ref{def:I}). Recall that $\delta = 2^{-32} \lambda^{-3}$.

\begin{lemma}\label{lem:stability}
Let $\lambda \ge 3$ and $n,k \in \N$, with $k \ge 2^{400} \lambda^{25}(\log n)^3$.  
We have
$$|\Lambda \setminus \Lambda^*| \le \frac{n^2}{k} \cdot |\cI| + \exp\bigg( - \frac{\delta k}{2^{10} \lambda} \bigg) {\lambda k / 2 \choose k}.$$
\end{lemma}

To prove Lemma~\ref{lem:stability}, we will successively refine $\Lambda \setminus \Lambda^*$, at each step showing that some subset with a particular property is small. The first step in the proof of Lemma~\ref{lem:stability} is the following stability lemma, which is an almost immediate consequence of Theorem~\ref{thm:M:stability}. 

\begin{lemma}\label{lem:stability:first}
Let $\lambda \ge 3$ and $n,k \in \N$, with $k \ge 2^{400} \lambda^{25}(\log n)^3$. 
There are at most
$$\exp\bigg( - \frac{\delta k}{2^9 \lambda} \bigg)  {\lambda k / 2 \choose k}$$
sets $A \in \Lambda$ such that
$$| A \setminus P | \ge \delta k$$
for every arithmetic progression $P$ of size $\lambda k/2$.
\end{lemma}

\begin{proof}
Set $\gamma = 2^{-9} \lambda^{-1} \delta = 2^{-41} \lambda^{-4}$, and observe that, since $k \ge 2^{400} \lambda^{25}(\log n)^3$, we have
$$2^{8} \lambda^{1/6} k^{-1/6} \sqrt{\log n} \le \gamma < 2^{-8}.$$
Therefore, by Theorem~\ref{thm:M:stability}, for all but at most
$$\exp\bigg( - \frac{\delta k}{2^9\lambda} \bigg) {\lambda k / 2 \choose k}$$
sets $A \in \Lambda$, there exists $T \subset A$, with $|T| \le (2^9 + 2^7 \lambda) \gamma k < \delta k$ (moving some elements of the progression given by the theorem into $T$), such that $A \setminus T$ is contained in an arithmetic progression of size $\lambda k / 2$, as required.
\end{proof}

The next step is to show that almost all sets $A \in \Lambda$ are contained in an arithmetic progression of length $3\lambda k/2$. Let us write $\cF$ for the family of sets $A \in \Lambda$ such that
$$A \subset \big\{ a + j d : -\lambda k/2 \le j \le \lambda k \big\} \qquad \text{and} \qquad \big| A \setminus \big\{ a + j d : 1 \le j \le \lambda k / 2 \big\} \big| \le \delta k$$
for some $a,d \in \Z$. Recall that we assume, for simplicity, that $\lambda k/2$ is an integer. 

\begin{lemma}\label{lem:stability:second}
Let $\lambda \ge 3$ and $n,k \in \N$, with $k \ge 2^{400} \lambda^{25}(\log n)^3$. 
Then 
$$|\Lambda \setminus \cF| \le \exp\bigg( - \frac{\delta k}{2^{10} \lambda} \bigg)  {\lambda k / 2 \choose k}.$$
\end{lemma}

\begin{proof}
Fix an arithmetic progression $P = \big\{ a + j d : 1 \le j \le \lambda k / 2 \big\}$. We will bound the number of sets $A \in \Lambda \setminus \cF$ with $|A \setminus P| \le \delta k$, and then sum over choices of $P$. We will then use Lemma~\ref{lem:stability:first} to count the remaining sets, and hence prove the lemma.

Note first that if $A \in \Lambda \setminus \cF$ and $|A \setminus P| \le \delta k$, then
$$A \not\subset P + P - P,$$ 
so let $Z := A \setminus (P + P - P)$ and choose an element $x \in Z$. We will first count the possible sets $A' := A \cap P$, and then (given $A'$) the choices for $B := A \setminus P$. Observe that
$$\big( x + A' \big) \cap \big( A' + A' \big) = \emptyset,$$
since $A' \subset P$, and that if $|A \setminus P| \le \delta k$, then $|x + A'| = |A'| \ge k - \delta k$. Since $A \in \Lambda$, it follows that
$$|A' + A'| \le \lambda k - \big( k - \delta k \big) \le \lambda k - k/2.$$ 
Hence, by Corollary~\ref{cor:prob:sumset} (applied with $m = k/2 \ge 2^{400} \lambda^{24}$), and using~\eqref{eq:prob:sumset:weak} and~\eqref{eq:binomial:usual:application}, it follows that, for each $b \le \delta k$, there are at most
$$\bigg( \frac{\lambda - 1}{\lambda} \bigg)^{k/4} {\lambda k / 2 \choose k - b} \le \exp\bigg( - \frac{k}{8 \lambda} \bigg) {\lambda k / 2 \choose k}$$
choices for the set $A' = A \cap P$ such that $|A'| = k - b$. 

To count the sets $B$ (given $A'$), we apply Ruzsa's covering lemma (Lemma~\ref{lem:RCL}) to the pair $(A',B)$ to obtain a set $X \subset B$, with $|X| \le |A' + B| / |A'| \le \lambda k / (k - b) \le 2\lambda$, such that $B \subset A' - A' + X$. Moreover, by the Pl\"unnecke--Ruzsa inequality (Lemma~\ref{lem:PR}),
$$|A' - A' + X| \le |X| \cdot |A' - A'| \le 2\lambda^3 k.$$ 
Hence, choosing $X$ first and then $B \setminus X$, and recalling that $b \le \delta k = 2^{-32} \lambda^{-3} k$, and that $k \ge 2^{400} \lambda^{25}(\log n)^3$, 
it follows that there are at most
$$n^{2\lambda} {2\lambda^3 k \choose b - 2\lambda} \le \exp\Big( \delta k \log\big( 2e\lambda^3 / \delta \big) + 2\lambda \log n \Big) \le \exp\big( \delta^{1/2} k \big)$$
choices for the set $B$, given a set $A'$ with $|A'| = k - b$. 

Combining the bounds above on the number of choices for $A'$ and $B$, it follows that the number of sets $A \in \Lambda$ with $Z$ non-empty is at most
$$\sum_{b = 1}^{\delta k} \exp\bigg(  \delta^{1/2} k - \frac{k}{8 \lambda} \bigg) {\lambda k / 2 \choose k} \le \, \exp\bigg( - \frac{k}{2^4 \lambda} \bigg) {\lambda k / 2  \choose k},$$
Summing over choices of $P$, and using Lemma~\ref{lem:stability:first} to bound the number of sets such that $| A \setminus P' | \ge \delta k$ for every arithmetic progression $P'$ of size $\lambda k/2$, the lemma follows.
\end{proof}

Finally, to bound $|\Lambda \setminus \Lambda^*|$ in terms of $|\cI|$, we need to map our arithmetic progression $P$ into the interval $[\lambda k / 2]$. Lemma~\ref{lem:stability} will follow from Lemma~\ref{lem:stability:second} and the following bound. 


\begin{lemma}\label{lem:red-to-int}
Let $\lambda \ge 3$ and $n,k \in \N$. Then
$$|\mathcal{F} \setminus \Lambda^*| \le \frac{n^2}{k} \cdot |\cI|.$$
\end{lemma}

\begin{proof}
We will define a function $\varphi \colon \cF \setminus \Lambda^* \to \cI$ such that $|\varphi^{-1}(S)| \le n^2/k$ for every $S \in \cI$, which will suffice to prove the lemma. To do so, let $A \in \cF \setminus \Lambda^*$, and choose $a,d \in \N$ such that
$$A \subset \big\{ a + j d : -\lambda k/2 \le j \le \lambda k \big\}$$
and such that the sets
\begin{equation}\label{eq:nonempty:sets}
\big\{ x \in A : x \le a \big\} \qquad \text{and} \qquad \big\{ x \in A : x > a + \lambda k d / 2 \big\}
\end{equation}
are both non-empty and together contain at most $\delta k$ elements. Indeed, to obtain such a pair, take the arithmetic progression given by the definition of $\cF$, and (recalling the definition~\eqref{def:Lambda:star} of $\Lambda^*$) translate it if necessary so that the sets in~\eqref{eq:nonempty:sets} are both non-empty. Now define
$$\varphi(A) := \big\{ j \in \Z \,:\, a + jd \in A \big\},$$
and observe that $\varphi(A) \subset \{ -\lambda k/2,\ldots,\lambda k\}$, and that
$$b\big( \varphi(A) \big) = \big| \big\{ x \in \varphi(A) : x \le 0 \big\} \big| + \big| \big\{ x \in \varphi(A) : x > \lambda k / 2 \big\} \big| \le \delta k.$$
Moreover, we have
$$r\big( \varphi(A) \big) = \max\big( \varphi(A) \big) - \min\big( \varphi(A) \big) - \frac{\lambda k}{2} > c(\lambda,\eps),$$
since $A \not\in \Lambda^*$, and hence $\varphi(A) \in \cI$, as required.

Finally, observe that $|\varphi^{-1}(S)|$ is bounded from above by the number of pairs $(a,d) \in \Z^2$ such that $A := \{ a + jd : j \in S \} \subset [n]$. For each set $S$ of size $k$ there are at most
$$\sum_{a = 1}^n \frac{n - a}{k - 1} \, \le \, \frac{n^2}{k}$$
such pairs $(a,d)$. Hence $|\varphi^{-1}(S)| \le n^2/k$, as claimed, and the lemma follows.
\end{proof}

We are now ready to prove Lemma~\ref{lem:stability}.

\begin{proof}[Proof of Lemma~\ref{lem:stability}]
By Lemmas~\ref{lem:stability:second} and~\ref{lem:red-to-int}, we have 
$$|\Lambda \setminus \Lambda^*| \le |\Lambda \setminus \cF| + |\cF \setminus \Lambda^*| \le \exp\bigg( - \frac{\delta k}{2^{10} \lambda} \bigg) {\lambda k / 2 \choose k} + \frac{n^2}{k} \cdot |\cI|,$$
as claimed.
\end{proof}

\section{Counting the sparse sets in \texorpdfstring{$\cI$}{I}}
\label{sec:spreadout}

Recall that, for any $A \subset \Z$,
$$b(A) = |A \setminus [\lambda k / 2]| \qquad \text{and} \qquad r(A) = \max(A) - \min(A) - \lambda k / 2,$$
and that $f(\lambda) = 2^{10} \lambda^3$, and (recalling Definition~\ref{def:I}) let us write 
$$\cS := \Big\{ A \in \cI \,:\, r(A) > f(\lambda) b(A) \Big\}$$
for the family of `sparse' sets in $\cI$. In this section we will bound the size of $\cS$, and hence prove the following quantitative version of Lemma~\ref{lem:sparse}. 

\begin{lemma}\label{lem:sparse:quant}
Let $\lambda \ge 3$ and $\eps \in (0,1)$, and let $k \in \N$. Then 
$$|\cS| \le \exp\bigg( - \frac{c(\lambda,\eps)}{2^9 \lambda^2} \bigg) {\lambda k/2 \choose k}.$$
\end{lemma}

\medskip

For each $B \subset \{-\lambda k / 2, \ldots, \lambda k \} \setminus [\lambda k/2]$, let us define\footnote{Note that we include sets of $\cI \setminus \cS$ in $\cG(B)$; we will not need to use the bound $r(A) > f(\lambda) b(A)$ when bounding the size of $\cG(B)$ (we use it only when counting the choices for the set $B$), and we shall also want to reuse our bounds on $|\cG(B)|$ in Section~\ref{sec:somewhat:dense}, below, where we will be dealing with dense sets.} 
\begin{equation}\label{eq:cG(B,B)}
\cG(B) := \big\{ A \in \cI \,:\, A \setminus [\lambda k / 2] = B \big\}.
\end{equation}
Recalling Definition~\ref{def:I}, observe that $\cG(B) = \emptyset$ if either $\min(B) > 0$ or $\max(B) \le \lambda k / 2$, and also if either $|B| > \delta k$ or $r(B) < c(\lambda,\eps)$ (note that $r(A) = r(B)$ for every $A \in \cG(B)$). 

We will deduce Lemma~\ref{lem:sparse:quant} from the following bound on the size of $\cG(B)$ by summing over $r \geq c(\lambda,\eps)$ and sets $B$ with $|B| < r / f(\lambda)$.

\begin{lemma}\label{lem:size_of_g}
If $B \subset \{-\lambda k / 2, \ldots, \lambda k \} \setminus [\lambda k/2]$, then 
$$|\cG(B)| \le \exp\bigg( - \frac{r}{2^6 \lambda^2} \bigg) {\lambda k/2 \choose k - b}$$
where $b = |B|$ and $r = r(B)$.
\end{lemma}

For each $A \in \cG(B)$, set $A' := A \setminus B$. The idea of the proof is simple: if $A'$ contains many elements close to its ends, then we can add these to $\min(B)$ and $\max(B)$, and obtain many elements of $A + A$ outside $[\lambda k]$. Therefore, either $A'+ A'$ misses many elements of $[\lambda k]$, in which case we can apply Corollary~\ref{cor:prob:sumset} to bound the number of choices, or it has few elements close to its ends, and it is straightforward to count sets $A'$ with this property. 

To be precise, define
\begin{equation}\label{def:Y}
Y := \big\{ x \le 0 : x - \min(B) \in A' \big\} \cup \big\{ x > \lambda k : x - \max(B) \in A' \big\},
\end{equation}
and set $m(B) := r(B) / 8\lambda$. The following bound follows from some simple counting. 

\begin{lemma}\label{lem:claim2}
If $B \subset \{-\lambda k / 2, \ldots, \lambda k \} \setminus [\lambda k/2]$, then there are at most
$$e^{-m(B)} {\lambda k/2 \choose k - b}$$
sets $A \in \cG(B)$ with $|Y| \le m(B)$.
\end{lemma}

\begin{proof}
We claim first that if $r := r(B) \ge \lambda k / 2$, then there are no such sets $A \in \cG(B)$. Indeed, if $\max(B) - \min(B) \ge \lambda k$ then for each $y \in A'$ either $y+ \min(B) \le 0$, or $y + \max(B) > \lambda k$, and therefore $|Y| \ge |A'|$. It follows that if $A \in \cG(B)$ with $|Y| \le m := m(B)$, then $m \ge |Y| \ge |A'| = k - b \ge k/4$, since $b(A) \le \delta k$ for every $A \in \cI$. But this implies that $r = 8\lambda m > \lambda k$, which is impossible. Let us therefore assume that $r < \lambda k / 2$. 

Now, the number of sets $A \in \cG(B)$ with $|Y| \le m$ is at most
\begin{equation}\label{eq:G:count:Y:small}
\sum_{\ell = 0}^{m} {r \choose \ell} {\lambda k/2 - r \choose k - b - \ell} \, \le \; \sum_{\ell = 0}^{m} \bigg( \frac{er}{\ell} \bigg)^\ell \left( 1 - \frac{2r}{\lambda k} \right)^{k-b-\ell} \left(\frac{2}{\lambda - 2} \right)^\ell {\lambda k/2 \choose k - b},
\end{equation}
where the inequality holds by~\eqref{eq:binomial:classic:too}. Now, observe that 
$$\left( 1 - \frac{2r}{\lambda k} \right)^{k-b-\ell} \le \left( 1 - \frac{2r}{\lambda k} \right)^{k/2} \le \exp\bigg( - \frac{r}{\lambda} \bigg) = e^{-8m},$$
since $b + \ell \le k/2$ and $r = 8\lambda m$, and that 
$$\sum_{\ell = 0}^{m} \bigg( \frac{er}{\ell} \cdot \frac{2}{\lambda - 2} \bigg)^\ell \le \, \sum_{\ell = 0}^{m} \bigg( \frac{2^4 e \lambda}{\lambda - 2} \cdot \frac{m}{\ell} \bigg)^\ell \le (m + 1) \bigg( \frac{2^4 e \lambda}{\lambda - 2} \bigg)^m \le \big( 2^7 e \big)^m,$$
since $r = 8\lambda m$ and $\lambda \ge 3$, and since $(C/x)^x$ is increasing for $x < C/e$. It follows that the right-hand side of~\eqref{eq:G:count:Y:small} (and hence the number of sets $A \in \cG(B)$ with $|Y| \le m$) is at most
$$\left( \frac{2}{e} \right)^{7m} {\lambda k/2 \choose k - b} \le e^{-m} {\lambda k/2 \choose k - b},$$
as claimed. 
\end{proof}

It remains to count sets $A \in \cG(B)$ with $|Y| > m$. To do so, set $X := A'+ A'$, and observe that $X$ and $Y$ are disjoint subsets of $A + A$. Since $|A+A| \le \lambda k$, it follows that
\begin{equation}\label{eq:G:key:obs}
\big| [\lambda k] \setminus X \big| \ge |Y|  > m(B).
\end{equation}
We will use Corollary~\ref{cor:prob:sumset} to count the sets with $| [\lambda k] \setminus X | \ge m(B)$.

\begin{lemma}\label{lem:claim1}
If $B \subset \{-\lambda k / 2, \ldots, \lambda k \} \setminus [\lambda k/2]$, then there are at most
$$\bigg( \frac{\lambda - 1}{\lambda} \bigg)^{m(B)/2} {\lambda k/2 \choose k - b}$$
sets $A \in \cG(B)$ with $| [\lambda k] \setminus X | \ge m(B)$.
\end{lemma}

\begin{proof}
We want to bound the number of sets $A' \subset [\lambda k / 2]$, with $|A'| = k - b$, such that $|[\lambda k] \setminus (A'+A')| \ge m := m(B)$. Recall that $|B| \le \delta k$ and $r(B) \ge c(\lambda,\eps)$ (otherwise $\cG(B)$ is empty), and note that therefore $m = r(B) / 8\lambda \ge 2^{400} \lambda^{24}$. It follows, by Corollary~\ref{cor:prob:sumset} and~\eqref{eq:prob:sumset:weak}, that there are at most 
$$\bigg( \frac{\lambda - 1}{\lambda} \bigg)^{m/2} {\lambda k/2 \choose k - b}$$
sets $A \in \cG(B)$ such that $|[\lambda k] \setminus (A'+ A')| \ge m$, as claimed. 
\end{proof}

We can now easily deduce the claimed upper bound on the size of $\cG(B)$.

\begin{proof}[Proof of Lemma~\ref{lem:size_of_g}]
By~\eqref{eq:G:key:obs}, $|\cG(B)|$ is at most the sum of the bounds in Lemmas~\ref{lem:claim2} and~\ref{lem:claim1}. Recalling that $m(B) = r(B) / 8\lambda$, this gives 
$$|\cG(B)| \le \big( e^{-m(B)} + e^{-m(B) / 2\lambda} \big) {\lambda k/2 \choose k - b} \le \exp\bigg( - \frac{r(B)}{2^5\lambda^2} \bigg) {\lambda k/2 \choose k - b},$$
as required.
\end{proof}

Lemma~\ref{lem:sparse:quant} is a straightforward consequence.

\begin{proof}[Proof of Lemma~\ref{lem:sparse:quant}]
Fix $b$ and $r$, and consider the sets $B \subset \{-\lambda k / 2, \ldots, \lambda k \} \setminus [\lambda k/2]$ with $|B| = b$ and $r(B) = r$. We may assume that $r > f(\lambda) b$ and $r \ge c(\lambda,\eps)$, since otherwise $\cG(B) \cap \cS = \emptyset$. The number of choices for $B$ (given $b$ and $r$) is therefore at most
$$\binom{r}{b} \le \big( 2^{10} e \lambda^3 \big)^{2^{-10} \lambda^{-3} r} \le \exp\bigg( \frac{r}{2^7 \lambda^2} \bigg)$$
since $r / b > f(\lambda) = 2^{10} \lambda^3$. By Lemma~\ref{lem:size_of_g}, it follows that
$$\big| \big\{ A \in \cS : b(A) = b, \, r(A) = r \big\} \big| \le \exp\bigg( - \frac{r}{2^7 \lambda^2} \bigg) \binom{\lambda k/2}{k-b} \le \exp\bigg( - \frac{r}{2^8 \lambda^2} \bigg) \binom{\lambda k/2}{k},$$
where the second inequality follows from~\eqref{eq:binomial:usual:application}, since $r/b > f(\lambda)$ and $\lambda \ge 3$. 

Summing over choices of $r \geq c(\lambda,\eps)$ and $b < r / f(\lambda)$, it follows that
$$|\cS| \le \sum_{r \geq c(\lambda,\eps)} \frac{r}{f(\lambda)} \exp\bigg( - \frac{r}{2^8 \lambda^2} \bigg) \binom{\lambda k/2}{k} \le \exp\bigg( - \frac{c(\lambda,\eps)}{2^9 \lambda^2} \bigg) \binom{\lambda k/2}{k},$$
as required.
\end{proof}

\section{Counting the moderately dense sets}\label{sec:somewhat:dense}

Recall from Definition~\ref{def:I} and~\eqref{def:b:and:r} the definitions of $b(A)$, $r(A)$ and $\cI$, and let us write
\begin{equation}\label{def:D}
\cD := \Big\{ A \in \cI \,:\, r(A) \le f(\lambda) b(A) \Big\}
\end{equation}
for the family of `dense' sets in $\cI$, where $f(\lambda) = 2^{10} \lambda^3$. In the next two sections we will prove the following quantitative version of Lemma~\ref{lem:dense}. 

\begin{lemma}\label{lem:dense:quant}
Let $\lambda \ge 3$ and $\eps \in (0,1)$, and let $k \in \N$. Then
$$|\cD| \le \exp\bigg( - \frac{c(\lambda,\eps)}{2^{20} \lambda^2} \bigg) \binom{\lambda k/2}{k}.$$
\end{lemma}

Let us fix $\lambda \ge 3$, $\eps \in (0,1)$ and $k \in \N$ until the end of the proof of Lemma~\ref{lem:dense:quant}. In this section, we will deal with some relatively easy cases using the method of the previous section. Observe that 
\begin{equation}\label{eq:b:lower:bound}
b(A) \ge \frac{c(\lambda,\eps)}{f(\lambda)} \ge 2^{550} \lambda^{29}
\end{equation}
for every $A \in \cD$, since $r(A) \ge c(\lambda,\eps)$ for every $A \in \cI$, and by the definition~\eqref{def:c} of $c(\lambda,\eps)$. 

For convenience, let us define, for each $b \in \N$ and $\mu \ge 1$, 
\begin{equation}\label{def:Dbmu}
\cD(b,\mu) := \Big\{ A \in \cD \,:\, |B| = b \,\text{ and }\, |(B+B) \setminus [\lambda k]| = \mu b, \text{ where } B = A \setminus [\lambda k / 2] \Big\}.
\end{equation}
The first step is to use Theorem~\ref{thm:M:counting} to bound the number of choices for $B = A \setminus [\lambda k / 2]$. We will use the following lemma several times in the proof of Lemma~\ref{lem:dense:quant}. 

\begin{lemma}\label{lem:count:Bs}
Let $b \in \N$ and $\mu > 2$. There are at most
\begin{equation}\label{lem:eq:choices:B}
e^{2\delta b} \bigg( \frac{\mu - 2}{2} \bigg)^{b} \bigg( \frac{\mu}{\mu - 2} \bigg)^{\mu b / 2}
\end{equation}
sets $B$ such that $B = A \setminus [\lambda k / 2]$ for some $A \in \cD(b,\mu)$. 
\end{lemma}

We will use the following observation in the proof of Lemma~\ref{lem:count:Bs}, and then again (several times) in the applications below. 

\begin{obs}\label{obs:log_concavity}
$$(x - 2) \cdot \left( \frac{x}{x - 2} \right)^{x/2} \le (y - 2) \left( \frac{y}{y - 2} \right)^{x/2}$$
for every $x,y > 2$.
\end{obs}

\begin{proof}
Set $q(x,y) := (x/y)^{x/2} \cdot \big( (y-2) / (x-2) \big)^{(x-2)/2}$, and observe that 
$$\log\big( q(x,y)^{2/x} \big) = \frac{2}{x} \cdot \log \frac{x}{y} + \frac{x-2}{x} \cdot \log \bigg( \frac{x (y-2)}{y (x-2)} \bigg) \le \log \bigg( \frac{2}{x} \cdot \frac{x}{y} + \frac{x-2}{x} \cdot \frac{x (y-2)}{y (x-2)} \bigg) = 0,$$
using the concavity of the $\log$ function.
\end{proof}

\begin{proof}[Proof of Lemma~\ref{lem:count:Bs}]
Set $B_1 := \{x \in B : x \le 0 \}$ and $B_2 := \{x \in B : x > \lambda k / 2 \}$, and recall from~\eqref{eq:b:lower:bound} that $b \ge 2^{550} \lambda^{29}$, and that $\delta = 2^{-32} \lambda^{-3}$. Observe first that, since $r(A) \le f(\lambda) b$ for each $A \in \cD(b,\mu)$, for each $i \in \{1,2\}$ there are at most
\begin{equation}\label{eq:choices:B:small}
{f(\lambda) b \choose b^{3/4}} \le \exp \big( b^{3/4} \log b \big) \le e^{\delta b}
\end{equation}
choices for the set $B_i$ with $|B_i| \le b^{3/4}$. Moreover, by Lemma~\ref{lem:F3k4}, if $|B_i + B_i| \le 2|B_i|$, then $B_i$ is contained in an arithmetic progression of size $|B_i| + 1$, and so in this case there are at most $r^3 \le 2^{30} \lambda^9 b^3 \le e^{\delta b}$ choices for $B_i$. Note that $\tfrac{\mu - 2}{2} \big( \tfrac{\mu}{\mu - 2} \big)^{\mu / 2} \ge 1$ for every $\mu > 2$, so these bounds suffice when either $|B_i| \le b^{3/4}$ or $|B_i + B_i| \le 2|B_i|$.

Now, set $b_i = |B_i|$ and $\mu_i b_i = |B_i + B_i|$, and suppose that $b_i \ge b^{3/4}$, and $\mu_i > 2$. In this case we will use Theorem \ref{thm:M:counting} to count the number of choices for $B_i$. To check the condition on $\mu_i$, observe that $B_i + B_i \subset [2\min(B),2\max(B)] \setminus [\lambda k]$, and therefore
\begin{equation}\label{eq:mui:bound}
\mu_i b_i \le 2 \cdot r(A) \le 2f(\lambda) b,
\end{equation} 
for every $A \in \cD(b,\mu)$, by~\eqref{def:b:and:r} and~\eqref{def:D}. Since $b_i \ge b^{3/4}$, and recalling that $b \ge 2^{550} \lambda^{29}$, it follows that\footnote{Using the bound $b_i \ge b^{3/4}$, the second inequality reduces to $b \ge 2^{74} f(\lambda)^2 \big( \log (f(\lambda) b) \big)^{6}$, which follows (with room to spare) from $b \ge 2^{550} \lambda^{29}$, since $f(\lambda) = 2^{10} \lambda^3$.}
$$\mu_i \le \frac{2f(\lambda) b}{b_i} \le 2^{-36} \frac{b_i}{\big( \log (f(\lambda)b) \big)^{3}}.$$
Hence, by Theorem \ref{thm:M:counting}, the number of choices for $B_i$ (given $b_i$ and $\mu_i$) is at most
\begin{equation}\label{eq:choices_b1}
\exp\Big(2^9 \mu_i^{1/6} b_i^{5/6} \log b_i \sqrt{\log (f(\lambda)b)}\Big) \binom{\mu_i b_i / 2}{b_i} \le e^{\delta b} \binom{\mu_i b_i / 2}{b_i},
\end{equation}
where the final inequality holds since $\mu_i^{1/6} b_i^{5/6} \le (\mu_i b_i)^{1/6} b^{3/4} \le 4\lambda \cdot b^{5/6}$, by~\eqref{eq:mui:bound}, so 
$$2^9 \mu_i^{1/6} b_i^{5/6} \log b_i \sqrt{\log (f(\lambda)b)} \le 2^{11} \lambda \cdot b^{5/6} (\log b)^2 \le \delta b,$$ 
since $\delta = 2^{-32} \lambda^{-3}$ and $b \ge 2^{550} \lambda^{29}$. 

Now, by Observations~\ref{obs:binomial} and~\ref{obs:log_concavity}, it follows that 
\begin{equation}\label{eq:choices_b2}
\binom{\mu_i b_i / 2}{b_i} \le \left( \frac{\mu_i - 2}{2} \cdot \left( \frac{\mu_i}{\mu_i - 2} \right)^{\mu_i/2} \right)^{b_i} \le \bigg( \frac{\mu - 2}{2} \bigg)^{b_i} \bigg( \frac{\mu}{\mu - 2} \bigg)^{\mu_i b_i / 2}.
\end{equation}
Since $\mu b = \mu_1 b_1 + \mu_2 b_2$, the lemma follows from~\eqref{eq:choices:B:small},~\eqref{eq:choices_b1} and~\eqref{eq:choices_b2}.
\end{proof}




We can now bound the number of sets $A \in \cD(b,\mu)$ such that $r(A) \ge 2^{11} \mu b$. 

\begin{lemma}\label{lem:r:vs:b}
Let $b \in \N$ and $\mu \ge 1$. If $r \ge 2^{11} \mu b$, then there are at most 
$$\exp\bigg( - \frac{r}{2^{7} \lambda^2} \bigg) \binom{\lambda k/2}{k}$$ 
sets $A \in \cD(b,\mu)$ with $r(A) = r$.
\end{lemma}

\begin{proof}
Observe first that if $\mu \le 2$, then $B$ is contained in two arithmetic progressions of combined size at most $|B| + 2$, by Lemma~\ref{lem:F3k4} (cf.~the proof of Lemma~\ref{lem:count:Bs}), and so in this case there are at most $r^6$ choices for~$B$. By Lemma~\ref{lem:size_of_g} and~\eqref{eq:binomial:usual:application}, it follows that there are at most
\begin{equation}\label{eq:rvsb:mu:small}
\sum_{B} |\cG(B)| \le r^6 \exp\bigg( - \frac{r}{2^6 \lambda^2} \bigg) \bigg( \frac{2}{\lambda - 2} \bigg)^b {\lambda k/2 \choose k}
\end{equation} 
sets $A \in \cD(b,\mu)$ with $r(A) = r$, where the sum is over sets with $|(B+B) \setminus [\lambda k]| \le 2|B| = 2b$ and $r(B) = r$. Now, since $r \ge 2^{11} b \ge 2^{561} \lambda^{29}$, by~\eqref{eq:b:lower:bound}, and $\lambda \ge 3$, we have\footnote{Indeed, if $\lambda \ge 4$ then note that $r^6 \le \exp\big( r / 2^7 \lambda^2 \big)$, and if $\lambda \le 4$ then note that $r^6 2^{r / 2^{11}} \le \exp\big( r / 2^{11} \big)$.} 
$$r^6 \bigg( \frac{2}{\lambda - 2} \bigg)^b \le \exp\bigg( \frac{r}{2^7 \lambda^2} \bigg),$$
and combining this with~\eqref{eq:rvsb:mu:small}, we obtain the claimed bound. 

Let us therefore assume from now on that $\mu > 2$. By Lemma \ref{lem:count:Bs} and Observation~\ref{obs:log_concavity}, it follows that there are at most 
\begin{equation}\label{eq:choices:B}
e^{2\delta b} \bigg( \frac{\lambda - 2}{2} \bigg)^{b} \bigg( \frac{\lambda}{\lambda - 2} \bigg)^{\mu b / 2}
\end{equation}
sets $B$ such that $B = A \setminus [\lambda k / 2]$ for some $A \in \cD(b,\mu)$. In order to count the sets $A$ for a given $B$, we will need to consider three cases. For each set $B$ that is counted in~\eqref{eq:choices:B}, set $m := m(B) = r(B) / 8\lambda$, and for each $A \in \cG(B)$, let $Y = Y(A)$ be the set defined in~\eqref{def:Y}.

\medskip
\noindent \textbf{Case 1:} $|Y| \le m$.
\medskip

By Lemma~\ref{lem:claim2} and~\eqref{eq:binomial:usual:application}, for each set $B$ there are at most 
$$e^{-m} {\lambda k/2 \choose k - b} \le e^{-m} \left( \frac{2}{\lambda - 2}\right)^b {\lambda k / 2 \choose k}$$
sets $A \in \cG(B)$ with $|Y| \le m$. Summing over sets $B$ as in~\eqref{eq:choices:B}, noting that if $r(B) = r$ then $\mu b \le 2^{-11} r \le 2^{-8} \lambda m$, and recalling that $\lambda \ge 3$, it follows that there are at most\footnote{Indeed, if $\lambda \le 4$ then $3^{\mu b} \le 3^{m/2^6} \le e^{m/2^5}$, and otherwise $\lambda / (\lambda - 2) \le e^{2/(\lambda-2)} \le e^{4/\lambda}$.} 
\begin{equation}\label{eq:r:vs:b:case1}
e^{2\delta b} \left(\frac{\lambda}{\lambda-2}\right)^{\mu b / 2} e^{-m} \binom{\lambda k/2}{k}\leq e^{-m/2} \binom{\lambda k/2}{k}
\end{equation}
sets $A \in \cD(b,\mu)$ with $r(A) = r$ and $|Y| \le m$. 


Counting the sets with larger $Y$ is somewhat more delicate, and we will need to partition into two cases, depending on the intersection of $Y$ with the set $B+B$. 

\medskip
\noindent \textbf{Case 2:} $|Y| \ge m$ and $|Y \cap (B+B)| \le m/2$. 
\medskip

In this case we will apply Corollary~\ref{cor:prob:sumset}. To do so, observe first that 
$$|A'+ A'| + |Y \cup (B+B) \setminus [\lambda k]| \le |A+A| \le \lambda k,$$
since $A' + A' \subset [\lambda k]$ and $Y \subset (A' + B) \setminus [\lambda k]$, by~\eqref{def:Y}. Recall that $|(B+B) \setminus [\lambda k]| = \mu b$ for each $A \in \cD(b,\mu)$, by~\eqref{def:Dbmu}.  Therefore, if $|Y| \ge m$ and $|Y \cap (B+B)| \le m/2$, then 
$$|[\lambda k] \setminus (A' + A')| \ge \mu b+m/2.$$ 
Moreover, if $\cD(b,\mu)$ is non-empty then $2^{400} \lambda^{24} \le b \le \delta k$, by~\eqref{eq:b:lower:bound} and Definition~\ref{def:I}, and since $\cD(b,\mu) \subset \cD \subset \cI$. 
Hence, by Corollary~\ref{cor:prob:sumset} and~\eqref{eq:binomial:usual:application}, it follows that for each set $B$ counted in~\eqref{eq:choices:B}, 
there are at most 
$$\exp\big( 2\delta \cdot (\mu b  + m/2) \big) \left(\frac{\lambda - 2}{\lambda}\right)^{\mu b / 2  + m / 4} \left( \frac{2}{\lambda - 2}\right)^b {\lambda k / 2 \choose k}$$
sets $A \in \cG(B)$ such that $|Y| \ge m$ and $|Y \cap (B+B)| \le m/2$. 

Summing over sets $B$, and using~\eqref{eq:choices:B}, it follows that there are at most
$$\exp\big( 2\delta \cdot (b + \mu b  + m/2) \big) \left(\frac{\lambda-2}{\lambda}\right)^{m/4} \binom{\lambda k/2}{k}$$
choices for $A$ in this case. Now, since $\mu b \le 2^{-8} \lambda m$ and $\delta = 2^{-32} \lambda^{-3}$, we have 
$$\exp\big( 2\delta \cdot (b + \mu b  + m/2) \big) \left(\frac{\lambda-2}{\lambda}\right)^{m/4} \le \exp\bigg( \delta \lambda m - \frac{m}{2\lambda} \bigg) \le \exp\bigg( - \frac{m}{4\lambda} \bigg),$$
and hence the number of sets $A$ with $|Y| \ge m$ and $|Y \cap (B+B)| \le m/2$ is at most
\begin{equation}\label{eq:r:vs:b:case2}
\exp\bigg( - \frac{m}{4\lambda} \bigg) \binom{\lambda k/2}{k} = \exp\bigg( - \frac{r}{2^5 \lambda^2} \bigg) \binom{\lambda k/2}{k}.
\end{equation}

\medskip

Finally, we count sets such that $Y$ has large intersection with $B+B$.  

\pagebreak
\noindent \textbf{Case 3:} $|Y| \ge m$ and $|Y \cap (B+B)| > m/2$.
\medskip

Let $B$ be such that $B = A \setminus [\lambda k / 2]$ for some $A \in \cD(b,\mu)$, and consider the set 
$$Z := \big\{ x \in [\lambda k/2] : x + \min(B) \in (B+B) \setminus [\lambda k] \;\text{ or }\; x + \max(B) \in (B+B) \setminus [\lambda k] \big\}.$$ 
Observe that $|A' \cap Z| > m/2$ and $|Z| \le |(B+B) \setminus [\lambda k]|$. It follows that the number of choices for $A'$ is at most 
$$\sum_{\ell > m/2} \binom{\mu b}{\ell}\binom{\lambda k / 2}{k - b - \ell} \le \sum_{\ell > m/2} \left( \frac{e\mu b}{\ell} \cdot \frac{2}{\lambda-2} \right)^\ell {\lambda k / 2 \choose k - b} \le 2^{-m} \left( \frac{2}{\lambda - 2}\right)^b {\lambda k / 2 \choose k},$$ 
where the inequalities follow from~\eqref{eq:binomial:usual:application} and the bounds $\mu b \le 2^{-8} \lambda m$ and $\lambda \ge 3$, which together imply that
$$\frac{2e\mu b}{m} \cdot \frac{2}{\lambda-2} \le \frac{e \lambda}{2^5(\lambda-2)} \le \frac{1}{4}.$$ 
By~\eqref{eq:choices:B}, and recalling again that $\mu b \le 2^{-8} \lambda m$, it follows that there are at most
\begin{equation}\label{eq:r:vs:b:case3}
e^{2\delta b} \left( \frac{\lambda}{\lambda-2} \right)^{\mu b / 2} 2^{-m} \binom{\lambda k/2}{k} \leq 2^{-m/2} \binom{\lambda k/2}{k} \le \exp\bigg( - \frac{r}{2^5 \lambda} \bigg) \binom{\lambda k/2}{k}
\end{equation}
choices for $A$ in this case. Summing~\eqref{eq:r:vs:b:case1},~\eqref{eq:r:vs:b:case2} and~\eqref{eq:r:vs:b:case3} gives the required bound on the number of sets $A \in \cD(b,\mu)$ with $r(A) = r$.
\end{proof}

It will be useful in the next section (which deals with the case $r \le 2^{11} \mu b$) to be able to assume that $\mu = \Theta(\lambda)$. The next lemma, which follows from Corollary~\ref{cor:prob:sumset}, provides a suitable bound on the size of $\cD(b,\mu)$ when this is not the case.   

\begin{lemma}\label{lem:big:or:small:sumset}
Let $b \in \N$. If $r \le 2^{11} \mu b$ and either $\mu \le 2$ or $\mu \not\in (\lambda / 2,2\lambda)$, then there are at most 
$$\exp\bigg( - \frac{r}{2^{16} \lambda} \bigg) \binom{\lambda k/2}{k}$$ 
sets $A \in \cD(b,\mu)$ with $r(A) = r$.
\end{lemma}
 
\begin{proof}
For each $A \in \cD(b,\mu)$, set $A' := A \cap [\lambda k / 2]$ and $B := A \setminus [\lambda k / 2]$, and observe that 
$$\big| [\lambda k] \setminus \big( A' + A' \big) \big| \,\ge\, \big| \big( B + B \big) \setminus [\lambda k] \big| \,=\, \mu b,$$ 
since $|A+A| \le \lambda k$. Hence, by Corollary~\ref{cor:prob:sumset} applied with $m = \mu b \ge 2^{400} \lambda^{24}$, and using~\eqref{eq:binomial:usual:application}, there are at most 
\begin{equation}\label{eq:counting:Aprime:big:or:small:sumset}
e^{2\delta \mu b} \left(\frac{\lambda - 2}{\lambda}\right)^{\mu b/2} \bigg( \frac{2}{\lambda - 2} \bigg)^b \binom{\lambda k/2}{k},
\end{equation}
choices for the set $A'$. 


Suppose first that $\mu > 2$, and recall from Lemma \ref{lem:count:Bs} that in this case there are at most 
\begin{equation}\label{eq:counting:B:big:or:small:sumset}
e^{2\delta b} \bigg( \frac{\mu - 2}{2} \bigg)^{b} \bigg( \frac{\mu}{\mu - 2} \bigg)^{\mu b / 2}
\end{equation}
sets $B$ with $B = A \setminus [\lambda k / 2]$ for some $A \in \cD(b,\mu)$. Moreover,  applying Observation~\ref{obs:log_concavity} with $x = \mu$ and $y = 2\lambda - 2$ gives
$$\bigg( \frac{\mu - 2}{2} \bigg)^b \bigg( \frac{\mu}{\mu - 2} \bigg)^{\mu b/2} \le (\lambda - 2)^b \bigg( \frac{\lambda - 1}{\lambda - 2} \bigg)^{\mu b/2}.$$
Thus, if $\mu \ge 2\lambda$ (and therefore $\mu \ge 6$), then the product of~\eqref{eq:counting:Aprime:big:or:small:sumset} and~\eqref{eq:counting:B:big:or:small:sumset} is at most\footnote{For the penultimate step, recall that $\delta \le 2^{-7} \lambda^{-1}$, and apply the inequality $2 \cdot e^{-x/2} < e^{-x/16}$, which holds for all $x \ge 2$, with $x = \mu / \lambda$.} 
$$e^{3\delta \mu b} \cdot 2^b \left(\frac{\lambda - 1}{\lambda}\right)^{\mu b/2} \binom{\lambda k/2}{k} \le \exp\bigg( - \frac{\mu b}{2^5 \lambda} \bigg) \binom{\lambda k/2}{k} \le \exp\bigg( - \frac{r}{2^{16} \lambda} \bigg) \binom{\lambda k/2}{k},$$
since $r \le 2^{11} \mu b$. 

Next, if $\mu > 2$ and $\lambda > 4$, then applying Observation~\ref{obs:log_concavity} with $x = \mu$ and $y = \lambda/2$ gives
$$\bigg( \frac{\mu - 2}{2} \bigg)^b \bigg( \frac{\mu}{\mu - 2} \bigg)^{\mu b/2} \le \bigg( \frac{\lambda - 4}{4} \bigg)^b \bigg( \frac{\lambda}{\lambda - 4} \bigg)^{\mu b/2}.$$
Thus, if $2 < \mu \le \lambda / 2$ (and therefore $\lambda > 4$), then the product of~\eqref{eq:counting:Aprime:big:or:small:sumset} and~\eqref{eq:counting:B:big:or:small:sumset} is at most 
$$e^{3\delta \mu b} \cdot \left(\frac{\lambda - 2}{\lambda - 4}\right)^{\mu b / 2} \bigg( \frac{\lambda - 4}{2\lambda - 4} \bigg)^b \binom{\lambda k/2}{k} \le 2^{-b/4} \binom{\lambda k/2}{k},$$
where the final step holds since $\delta = 2^{-32} \lambda^{-3}$, $\mu \le \lambda / 2$, 
and 
$$2^4 \cdot \bigg( \frac{\lambda - 2}{\lambda - 4} \bigg)^{\lambda} \bigg( \frac{\lambda - 4}{2\lambda - 4} \bigg)^4 = \bigg( 1 + \frac{2}{\lambda - 4} \bigg)^{\lambda - 4} \le e^{2}.$$
Since $r \le 2^{11} \mu b \le 2^{10} \lambda b$, it follows that if $2 < \mu \le \lambda / 2$ then there are at most 
$$\exp\bigg( - \frac{r}{2^{13} \lambda} \bigg) \binom{\lambda k/2}{k}$$
sets $A \in \cD(b,\mu)$ with $r(A) = r$.

Finally, if $\mu \le 2$ then $B$ is contained in two arithmetic progressions of combined size at most $|B| + 2$, by Lemma~\ref{lem:F3k4}, and so in this case there are at most $r^6 \le 2^{72} b^6 \le e^{\delta b}$ choices for~$B$, by~\eqref{eq:b:lower:bound}. Noting that $\mu b \ge 2b - 2$, it follows from~\eqref{eq:counting:Aprime:big:or:small:sumset} that there are at most
$$e^{5\delta b} \left(\frac{\lambda - 2}{\lambda}\right)^{b-1} \bigg( \frac{2}{\lambda - 2} \bigg)^b \binom{\lambda k/2}{k} \le \exp\bigg( - \frac{r}{2^{14}} \bigg) \binom{\lambda k/2}{k}$$
choices for $A$, where the last inequality holds since $\lambda \ge 3$ 
and $r \le 2^{11} \mu b \le 2^{12} b$.
\end{proof}



\section{Counting the very dense sets with containers}\label{sec:dense}

It remains to bound the size of the family\footnote{Recall that the family $\cD(b,\mu)$ was defined in~\eqref{def:Dbmu}.} 
\begin{equation}\label{def:Dstar}
\cD^*(b,\mu) := \Big\{ A \in \cD(b,\mu) : r(A) \le 2^{11} \mu b \Big\}
\end{equation}
of \emph{very dense} sets, for each $\mu > 2$ with $\lambda/2 \le \mu \le 2\lambda$. To do so, we will once again use the container theorem from~\cite{C19} (Theorem~\ref{thm:M:containers}), but this time our application of it will be rather different. Recall first, from Lemma~\ref{lem:middle:covered}, that the `missing' set $M(A) := [\lambda k] \setminus (A+A)$ is typically contained near the ends of $[\lambda k]$ (see Lemma~\ref{lem:hitmiddle}, below). We will use Corollary~\ref{cor:M:containers} to find a family of $2^{o(b)}$ containers $(C,D)$ for the parts of $A$ `close' to the endpoints, and for $M(A)$ (see Corollary~\ref{cor:M:containers:app}). We will then, in Lemma~\ref{lem:finalcount}, bound the number of sets $A \in \cD^*(b,\mu)$ corresponding to each container. Our bound decreases exponentially with $b$, and we will therefore be able to take a union bound over containers. 




To state the version of Corollary~\ref{cor:M:containers} we will use, we need a little additional notation. First, for each $b \in \N$, set $Y(b) := Y_1 \cup Y_2$ and $X(b) := (Y_1 + Y_1) \cup (Y_2 + Y_2)$, where
$$Y_1 := \big\{ 0,\ldots, 2^{18} \lambda^2 b \big\}, \qquad \text{and} \qquad Y_2 := \big\{ \lambda k / 2 - 2^{18} \lambda^2 b, \ldots, \lambda k / 2 \big\},$$
Moreover, define $M(A) := [\lambda k] \setminus (A+A)$ and 
$$\cT(b) := \big\{ A \in \cI \,:\, b(A) = b \,\text{ and }\, M(A) \subset X(b) \big\}.$$
As we will see below (see Lemma~\ref{lem:hitmiddle}), this family contains almost all of $\cD^*(b,\mu)$.

Our key tool in this section will be the following immediate consequence of Corollary~\ref{cor:M:containers}. 

\begin{cor}\label{cor:M:containers:app}
For each $b \in \N$, there exists a family $\cB(b) \subset 2^{X(b)} \times 2^{Y(b)}$ of size at most 
$$\exp\big( 2^{50} \lambda^2 b^{7/8} \big)$$ 
such that:
\begin{itemize}
\item[$(a)$] For each $A \in \cT(b)$, there exists $(C, D) \in \cB(b)$ with $M(A) \subset C$ and $A \cap Y(b) \subset D$.
\item[$(b)$] For every $(C,D) \in \cB(b)$, 
$$|D| \leq \max\bigg\{ |Y(b)| - \frac{|C|}{2} + |Y(b)|^{5/6}, \, \frac{3|Y(b)|}{\log |Y(b)|} \bigg\}.$$
\end{itemize}
\end{cor}

\begin{proof}
We apply Corollary~\ref{cor:M:containers} with $S_1 = Y_1$, $S_2 = Y_2$ and $\gamma = |Y(b)|^{-1/6} / 4$. The bound on the size of $\cB(b)$ follows from~\eqref{eq:number:of:containers}, since $\log |Y(b)| \le 2^{6} (\lambda^2 b)^{1/36}$ and 
$$2^{17} \gamma^{-2} \sqrt{|Y(b)|} = 2^{21} |Y(b)|^{5/6} \le 2^{21} \big( 2^{20} \lambda^2 b \big)^{5/6} \le 2^{41} \lambda^{5/3} b^{5/6},$$
where in both cases we used the bound $|Y(b)| \le 2^{20} \lambda^2 b$.
 
The bound on $|D|$ for each $(C,D) \in \cB(b)$ follows from~\eqref{eq:size:of:D}. Finally, for each $A \in \cT(b)$ we apply Corollary~\ref{cor:M:containers}$(a)$ with $U := A \cap Y(b)$ and $W := M(A) \subset X(b) \setminus (U + U)$. It follows that there exists $(C, D) \in \cB(b)$ such that $M(A) \subset C$ and $A \cap Y(b) \subset D$, as claimed.
\end{proof}

Recall from~\eqref{eq:b:lower:bound} that $b(A) \ge 2^{550} \lambda^{29}$ for every $A \in \cD^*(b,\mu) \subset \cD$. Since $\delta = 2^{-32} \lambda^{-3}$, it follows that 
\begin{equation}\label{eq:Yb:small}
|Y(b)|^{5/6} \le \delta b.
\end{equation}
In the calculations below, we will also need the inequalities
\begin{equation}\label{eq:another:binomial:inequality}
{a + c \choose b} \le \bigg( 1 + \frac{c}{a-b} \bigg)^b {a \choose b} \qquad \text{and} \qquad {a - c \choose b - c} \le \bigg( \frac{b}{a} \bigg)^c {a \choose b}
\end{equation}
Before bounding the number of sets in each container, let's first observe that, by our choice of $X(b)$, most members of $\cD^*(b,\mu)$ are also in $\cT(b)$. 

\begin{lemma}\label{lem:hitmiddle}
For each $b \le \delta k$ and $\mu \le 2\lambda$, 
there are at most
\begin{equation}\label{eq:hitmiddle}
e^{-b} {\lambda k / 2 \choose k}
\end{equation}
sets $A \in \cD^*(b,\mu)$ such that $M(A) \not\subset X(b)$.
\end{lemma}

\begin{proof}
Recalling~\eqref{def:Dstar}, let $A$ be a uniformly random $k$-subset of $L := [-2^{11} \mu b,\lambda k / 2 + 2^{11} \mu b]$, and observe that
$$\Pr\big( M(A) \not\subset X(b) \big) \le \Pr\Big( \big\{ M'+ 1, \dotsc, \lambda k - M' - 1 \big\} \not\subset A + A \Big),$$
where $M' := 2^{19} \lambda^2 b$, by the definitions of $M(A) = [\lambda k] \setminus (A+A)$ and $X(b)$. By Lemma~\ref{lem:middle:covered} (applied with $n = \lambda k / 2 + 2^{12} \mu b + 1$ and $M = M' + 2^{12} \mu b + 2$), it follows that 
$$\Pr\big( M(A) \not\subset X(b) \big) \le \frac{8}{p^2} \cdot \big(1 - p^2\big)^{M/2},$$
where $p = k \big( \lambda k / 2 + 2^{12} \mu b + 1 \big)^{-1}$. Now, recall that $b \le \delta k$ and $\mu \le 2\lambda$, and observe that therefore $p \ge 1/\lambda$. 
Since $M \ge M' = 2^{19} \lambda^2 b$, it follows that
$$\Pr\big( M(A) \not\subset X(b) \big) \le 8\lambda^2 \cdot e^{-M/2\lambda^2} \le \exp\big( - 2^{17} b \big),$$
since $b \ge 2^{550} \lambda^{29}$. In order to deduce a bound on the number of sets such that $M(A) \not\subset X(b)$, we simply need to multiply by the total number of $k$-subsets of $L$. 
There are at most
$${\lambda k / 2 + 2^{12} \mu b + 1 \choose k} \le \bigg( 1 + \frac{2^{13} \mu b + 2}{(\lambda - 2) k} \bigg)^k {\lambda k / 2 \choose k} \le \exp\big( 2^{16} b \big) {\lambda k / 2 \choose k}$$
such sets, where the first inequality holds by~\eqref{eq:another:binomial:inequality}, 
and the second because $\lambda \ge 3$ and $\mu \le 2\lambda$. Hence, there are at most 
$$\exp\big( - 2^{17} b + 2^{16} b \big) {\lambda k / 2 \choose k} \le e^{-b} {\lambda k / 2 \choose k}$$
sets $A \in \cD^*(b,\mu)$ with $M(A) \not\subset X(b)$, as claimed.
\end{proof}

\pagebreak

To deduce Lemma~\ref{lem:dense:quant} from Corollary~\ref{cor:M:containers:app}, we will need to bound the size of the containers in $\cB(b)$. The following lemma provides the bound we need. 

\begin{lemma}\label{lem:finalcount}
Let $b \le \delta k$ and $\mu > 2$, with $\lambda / 2 \le \mu \le 2\lambda$. For each $(C,D) \in \cB(b)$,  
there are at most  
$$e^{- b / 32\lambda} \binom{\lambda k/2}{k}$$
sets $A \in \cD^*(b,\mu)$ such that $M(A) \subset C$ and $A \cap Y(b) \subset D$.
\end{lemma} 

In the proof of Lemma~\ref{lem:finalcount}, we will need the following binomial inequalities, whose (straightforward) proofs are postponed to Appendix~\ref{app:binomial}. 
Set $\alpha := 2^{25} \lambda^2 \delta = 2^{-7} \lambda^{-1}$. 

\begin{obs}\label{obs:choices:boring}
Let $b \le \delta k$ and $\mu > 2$, with $\lambda / 2 \le \mu \le 2\lambda$, and let $s \le t \le 2^{22} \lambda^2 b$. Then
$${\lambda k / 2 - \mu b - s \choose k - b - s} \le e^{\alpha b} \bigg( \frac{\lambda - 2}{\lambda} \bigg)^{\mu b} \bigg( \frac{2}{\lambda - 2} \bigg)^{b} {\lambda k / 2 - s \choose k - s},$$
and
$${\lambda k / 2 - \mu b/2 - s/ 2 - t + \delta b \choose k - b - s} \le e^{\alpha b} \bigg( \frac{\lambda - 2}{\lambda} \bigg)^{\mu b/2} \bigg( \frac{2}{\lambda - 2} \bigg)^{b} \binom{\lambda k / 2 - s/ 2 - t}{k - s}.$$
\end{obs}

We are now ready to prove our key lemma.

\begin{proof}[Proof of Lemma~\ref{lem:finalcount}]
First, by Lemma~\ref{lem:count:Bs} and Observation~\ref{obs:log_concavity}, 
there are at most
\begin{equation}\label{eq:Csmall:Bcount}
e^{2\delta b} \bigg( \frac{\mu - 2}{2} \bigg)^{b} \bigg( \frac{\mu}{\mu - 2} \bigg)^{\mu b / 2} \le e^{2\delta b} \bigg( \frac{\lambda - 2}{2} \bigg)^{b} \bigg( \frac{\lambda}{\lambda - 2} \bigg)^{\mu b / 2}
\end{equation}
sets $B$ such that $B = A \setminus [\lambda k / 2]$ for some $A \in \cD^*(b,\mu) \subset \cD(b,\mu)$. Fix such a set $B$, let $A \in \cD^*(b,\mu)$ with $B = A \setminus [\lambda k / 2]$, and recall that $M(A) = [\lambda k] \setminus (A+A)$. Note that $|M(A)| \ge \mu b$, since $|A+A| \le \lambda k$ and $|(B+B) \setminus [\lambda k]| = \mu b$. Now, define\footnote{To avoid any possible confusion, we emphasize that $\tilde{C}$ is the union of two shifted copies of the set $C$.}
\begin{equation}\label{def:CDtilde}
\tilde{C} := C - \{ \min(B),\max(B) \} \qquad \text{and} \qquad \tilde{D} := \big( [\lambda k /2] \setminus Y(b) \big) \cup D,
\end{equation}
and observe that if $A \cap Y(b) \subset D$, then $A' \subset \tilde{D}$, where (as usual) $A' = A \cap [\lambda k / 2]$. Set 
$$S := \tilde{C} \cap A' \qquad \textup{and} \qquad T := \tilde{C} \cap \tilde{D},$$
and observe that $S \subset T$, and that for each $x \in S$, either $x + \max(B)$ or $x + \min(B)$ is contained in $C \cap (A+A)$. Moreover, the sets $S + \max(B)$ and $S + \min(B)$ are disjoint, since $S \subset [\lambda k / 2]$ and $\max(B) - \min(B) > \lambda k /2$. It follows that if $M(A) \subset C$, then 
\begin{equation}\label{eq:C:bound:muS}
|C| \ge |M(A)| + |S| \ge \mu b + |S|.
\end{equation}

For each $s,t \in \N$, let us write $g(s,t)$ for the number of sets $A \in \cD^*(b,\mu)$ such that $M(A) \subset C$ and $A \cap Y(b) \subset D$, and such that
$$|S| = s \qquad \textup{and} \qquad |T| = t.$$
Since $S \subset T$, we have at most ${t \choose s}$ choices for $S$. We will bound $g(s,t)$ in two different ways, depending on the values of $s$ and $t$. 

\medskip
\noindent \textbf{Claim:} If $s \le b/16$ and $t \le \lambda b$, then
\begin{equation}\label{eq:gst:claim}
g(s,t) \le e^{-b/8} {\lambda k / 2 \choose k}.
\end{equation}

\begin{proof}[Proof of claim]
In this case we will use the bound
\begin{equation}\label{eq:DminusC:bound1}
|\tilde{D} \setminus \tilde{C}| \le \frac{\lambda k}{2} - |C| \le \frac{\lambda k}{2} - \mu b - s.
\end{equation}
The second inequality is~\eqref{eq:C:bound:muS}, and therefore, recalling that $\tilde{D} \subset [\lambda k / 2]$, to prove~\eqref{eq:DminusC:bound1} it will suffice to show that
\begin{equation}\label{eq:Ctilde:vs:C}
\big| \tilde{C} \cap [\lambda k / 2] \big| \ge |C|.
\end{equation}
To prove~\eqref{eq:Ctilde:vs:C}, observe first that $2^{19} \lambda^2 b \le \lambda k/8$, since $b \le \delta k$ and $\delta = 2^{-32} \lambda^{-3}$, and therefore
\begin{equation}\label{eq:MA:Xb:subset}
C \subset X(b) \subset \big\{ 0,\ldots, \lambda k/8 \big\} \cup \big\{ 7\lambda k / 8, \ldots, \lambda k \big\}.
\end{equation}
Moreover, $r(A) \le 2^{11} \mu b \le 2^{12} \lambda \cdot \delta k \le \lambda k/8$ for every $A \in \cD^*(b,\mu)$, and therefore 
\begin{equation}\label{eq:minB:maxB:bounds}
- \lambda k / 8 < \min(B) \le 0 \qquad \text{and} \qquad \lambda k / 2 < \max(B) <  5\lambda k / 8.
\end{equation}
It follows from~\eqref{eq:MA:Xb:subset} and~\eqref{eq:minB:maxB:bounds} that $\big| \tilde{C} \cap [\lambda k / 2] \big| \ge |C|$, as claimed.

Now, recalling that $A' \subset \tilde{D}$ and $S = \tilde{C} \cap A'$, it follows from~\eqref{eq:Csmall:Bcount} and~\eqref{eq:DminusC:bound1} that 
$$g(s,t) \le e^{2\delta b} \bigg( \frac{\lambda - 2}{2} \bigg)^{b} \bigg( \frac{\lambda}{\lambda - 2} \bigg)^{\mu b / 2} {\lambda k / 2 - \mu b - s \choose k - b - s} {t \choose s}.$$
Observe that $s \le t \le |\tilde{C}| \le 2|C| \le 2 |X(b)| \le 2^{22} \lambda^2 b$. Thus, by Observation~\ref{obs:choices:boring}, we have
$${\lambda k / 2 - \mu b - s \choose k - b - s} \le e^{\alpha b} \bigg( \frac{\lambda - 2}{\lambda} \bigg)^{\mu b} \bigg( \frac{2}{\lambda - 2} \bigg)^{b} {\lambda k / 2 - s \choose k - s},$$
and therefore, by~\eqref{eq:another:binomial:inequality}, 
$$g(s,t) \le e^{2\alpha b} \bigg( \frac{\lambda - 2}{\lambda} \bigg)^{\mu b/2} {\lambda k / 2 - s \choose k - s} {t \choose s} \le e^{2\alpha b} \bigg( \frac{\lambda - 2}{\lambda} \bigg)^{\mu b/2} \bigg( \frac{2}{\lambda} \cdot \frac{et}{s} \bigg)^{s} {\lambda k / 2 \choose k}.$$
Since $s \le b/16$ and $t \le \lambda b$, and recalling that $\mu \ge \lambda/2$, it follows that
$$g(s,t) \le e^{2\alpha b} \cdot e^{-b/2} \cdot \big( 32e \big)^{b/16}  {\lambda k / 2 \choose k} \le e^{-b/8} {\lambda k / 2 \choose k},$$
as claimed.
\end{proof}

\pagebreak

We may therefore assume that either $s \ge b/16$ or $t \ge \lambda b$. In this case observe that $|\tilde{D}| = \lambda k / 2 - |Y(b)| + |D|$, by~\eqref{def:CDtilde} (and since $D \subset Y(b) \subset [\lambda k / 2]$), and therefore
\begin{equation}\label{eq:DminusC:bound2}
|\tilde{D}| \, \le \, \frac{\lambda k}{2} - \frac{|C|}{2} + |Y(b)|^{5/6} \, \le \, \frac{\lambda k - \mu b - s}{2} + \delta b,
\end{equation}
by Corollary~\ref{cor:M:containers:app}$(b)$, together with~\eqref{eq:Yb:small} and~\eqref{eq:C:bound:muS}. Since $A' \subset \tilde{D}$ and $S = \tilde{C} \cap A'$, it follows from~\eqref{eq:Csmall:Bcount} and~\eqref{eq:DminusC:bound2} that
$$g(s,t) \le e^{2\delta b} \bigg( \frac{\lambda - 2}{2} \bigg)^{b} \bigg( \frac{\lambda}{\lambda - 2} \bigg)^{\mu b / 2} {\lambda k / 2 - \mu b/2 - s/ 2 - t + \delta b \choose k - b - s} {t \choose s}.$$
Since $s \le t \le 2^{22} \lambda^2 b$, it follows by Observation~\ref{obs:choices:boring} that
\begin{equation}\label{eq:finalcount:general}
g(s,t) \le e^{2\alpha b} {\lambda k / 2 - s/2 - t \choose k - s} {t \choose s}. 
\end{equation}
Now, if $s \ge b/16$ then, by~\eqref{fact:binomial_classic}, we have
$${\lambda k / 2 - s/2 - t \choose k - s} {t \choose s} \le {\lambda k / 2 - s/2 \choose k} \le \bigg( \frac{\lambda - 2}{\lambda} \bigg)^{s/2} {\lambda k / 2 \choose k} \le e^{-b/16\lambda} {\lambda k / 2 \choose k}.$$
On the other hand, if $s \le b/16$ and $t \ge \lambda b$, then using~\eqref{eq:another:binomial:inequality} and~\eqref{fact:binomial_classic}, and noting that $t - s/2 \ge t/2$, we obtain
$${\lambda k / 2 - s/ 2 - t \choose k - s} {t \choose s} \le \bigg( \frac{2}{\lambda} \cdot \frac{et}{s} \bigg)^{s} {\lambda k / 2 + s/2 - t \choose k} \le \bigg( \frac{2et}{\lambda s} \bigg)^{s} \bigg( \frac{\lambda - 2}{\lambda} \bigg)^{t/2} {\lambda k / 2 \choose k}.$$
Now, observe that (for $s \le b/16$ and $t \ge \lambda b$) the right-hand side is increasing in $s$ and decreasing in $t$, since $2et / \lambda s \ge 32e$ (and by simple calculus). It follows that
$$\bigg( \frac{2et}{\lambda s} \bigg)^{s} \bigg( \frac{\lambda - 2}{\lambda} \bigg)^{t/2} \le \big( 32e \big)^{b/16} \cdot e^{-b} \le e^{-b/2}.$$
Combining these bounds, and recalling that $\alpha = 2^{-7} \lambda^{-1}$, we obtain
\begin{equation}\label{eq:gst:finalbound}
g(s,t) \le e^{2\alpha b} \Big( e^{-b/16\lambda} + e^{-b/2} \Big) {\lambda k / 2 \choose k} \le e^{-b/24\lambda} {\lambda k / 2 \choose k}.
\end{equation}
Finally, summing the bounds~\eqref{eq:gst:claim} and~\eqref{eq:gst:finalbound} over $s \le t \le 2^{22} \lambda^2 b$, and recalling that $b \ge 2^{550} \lambda^{29}$, we obtain the claimed bound. 
\end{proof}

We are finally ready to prove Lemma~\ref{lem:dense:quant}.

\pagebreak

\begin{proof}[Proof of Lemma~\ref{lem:dense:quant}]
Let us fix $b,r \in \N$ and $\mu \ge 1$, and bound the number of sets $A \in \cD(b,\mu)$ with $r(A) = r$. Recall first that if $r \ge 2^{11} \mu b$ then, by Lemma~\ref{lem:r:vs:b}, there are at most
$$\exp\bigg( - \frac{r}{2^{7} \lambda^2} \bigg) \binom{\lambda k/2}{k}$$ 
such sets, and if $r \le 2^{11} \mu b$ and either $\mu \le 2$, $\mu \le \lambda/2$ or $\mu \ge 2\lambda$, then by Lemma~\ref{lem:big:or:small:sumset} there are at most
$$\exp\bigg( - \frac{r}{2^{16} \lambda} \bigg) \binom{\lambda k/2}{k}$$ 
such sets. Now, if $r \le 2^{11} \mu b$ and $\lambda / 2 \le \mu \le 2\lambda$, then by Lemma~\ref{lem:hitmiddle} there are at most
$$e^{-b} {\lambda k / 2 \choose k} \le \exp\bigg( - \frac{r}{2^{12} \lambda} \bigg) \binom{\lambda k/2}{k}$$
such sets that are not in $\cT(b)$. Moreover, by Corollary~\ref{cor:M:containers:app}, there exists a family $\cB(b)$ of size at most 
$$\exp\big( 2^{50} \lambda^2 b^{7/8} \big)$$ 
such that for every $A \in \cT(b)$, there exists $(C, D) \in \cB(b)$ with $M(A) \subset C$ and $A \cap Y(b) \subset D$. Finally, by Lemma~\ref{lem:finalcount}, for each $(C,D) \in \cB(b)$ there are at most  
$$e^{- b / 32\lambda} \binom{\lambda k/2}{k} \le \exp\bigg( - \frac{r}{2^{18} \lambda^2} \bigg) \binom{\lambda k/2}{k}$$
sets $A \in \cT(b) \cap \cD^*(b,\mu)$ such that $M(A) \subset C$ and $A \cap Y(b) \subset D$.

Combining these bounds, it follows that there are at most
$$\exp\big( 2^{50} \lambda^2 b^{7/8} \big) \exp\bigg( - \frac{r}{2^{18} \lambda^2} \bigg) \binom{\lambda k/2}{k}$$ 
sets $A \in \cD(b,\mu)$ with $r(A) = r$. Now, summing over choices of $b \le r$ and $\mu \le 2r/b$ such that $\mu b \in \N$, and recalling that $r \ge 2^{560} \lambda^{32}$, it follows that there are at most
$$\exp\bigg( - \frac{r}{2^{19} \lambda^2} \bigg) \binom{\lambda k/2}{k}$$ 
sets $A \in \cD$ with $r(A) = r$. 

Finally, summing over $r \ge c(\lambda,\eps)$, we deduce that 
$$|\cD| \le \exp\bigg( - \frac{c(\lambda,\eps)}{2^{20} \lambda^2} \bigg) \binom{\lambda k/2}{k},$$
as claimed.
\end{proof}


\section{The proof of Theorem~\ref{thm:structure}}\label{proof:sec}
 
In this section we will prove the following quantitative version of Theorem~\ref{thm:structure}, which allows us to control the typical structure of $A$ when $\lambda = k^{o(1)}$. Recall that $\delta = 2^{-32} \lambda^{-3}$.

\begin{theorem}\label{thm:structure:quant}
Let $\lambda \ge 3$ and $n,k \in \N$ be such that $k \ge 2^{400} \lambda^{25}(\log n)^3$,  
and let $\eps > e^{-\delta^2 k}$. Let~$A \subset [n]$ be chosen uniformly at random from the sets with $|A| = k$ and $|A+A| \le \lambda k$. Then there exists an arithmetic progression $P$ with
$$A \subset P \qquad \text{and} \qquad |P| \le \frac{\lambda k}{2} + c(\lambda,\eps)$$ 
with probability at least $1 - \eps$.
\end{theorem}

There is only one piece still missing in the proof of Theorem~\ref{thm:structure:quant}: a lower bound on the size of the set 
$$\Lambda = \big\{ A \subset [n] \,:\, |A| = k \, \text{ and } \, |A+A| \le \lambda k \big\}.$$ 
The following very simple bound will suffice for our current purposes; a stronger lower bound (at least, for large $\lambda$) will be proved in Section~\ref{lower:sec}. 

\begin{lemma}\label{lem:easy:lower}
Let $\lambda \ge 3$ and $n,k \in \N$, with $\lambda k \le n$. Then
$$\big| \big\{ A \subset [n] : |A| = k, \, |A+A| \le \lambda k  \big\} \big| \ge \frac{1}{\lambda^3} \cdot \frac{n^2}{k} \binom{\lambda k/2}{k}.$$
\end{lemma}

\begin{proof}
We consider, for each arithmetic progression $P$ of length $\lambda k / 2$ in $[n]$, all subsets $A \subset P$ of size $k$ containing both endpoints of $P$. All of these sets are distinct, and all satisfy $|A+A| \le \lambda k$. There are at least $n^2/2\lambda k$ choices for the arithmetic progression, and therefore
$$|\Lambda| \ge \frac{n^2}{2\lambda k} {\lambda k / 2 - 2 \choose k - 2} \ge \frac{n^2}{\lambda^3 k} {\lambda k / 2 \choose k},$$
as claimed, where the final step follows since ${a \choose b} = \frac{a(a-1)}{b(b-1)} {a - 2 \choose b - 2}$.
\end{proof}

We can now deduce Theorem~\ref{thm:structure:quant} from Lemmas~\ref{lem:stability},~\ref{lem:sparse:quant},~\ref{lem:dense:quant} and~\ref{lem:easy:lower}.  

\begin{proof}[Proof of Theorem~\ref{thm:structure:quant}]
For simplicity, we will assume that $\lambda k \le n$; the case $\lambda k > n$ is dealt with in Appendix~\ref{app:boring}. 
By Lemma~\ref{lem:stability}, and since $\eps > e^{-\delta^2 k}$, we have
$$|\Lambda \setminus \Lambda^*| \le \frac{n^2}{k} \cdot |\cI| + \exp\bigg( - \frac{\delta k}{2^{10} \lambda} \bigg) {\lambda k / 2 \choose k} \le \frac{n^2}{k} \cdot |\cI| + \frac{\eps}{2\lambda^3} {\lambda k / 2 \choose k}.$$
Now, by Lemmas~\ref{lem:sparse:quant} and~\ref{lem:dense:quant}, and recalling that $\cS \cup \cD = \cI$, we have 
$$|\cI| = |\cS| + |\cD| \le 2 \cdot \exp\bigg( - \frac{c(\lambda,\eps)}{2^{20} \lambda^2} \bigg) {\lambda k/2 \choose k} \le \frac{\eps}{2\lambda^3} {\lambda k/2 \choose k}$$
since $c(\lambda,\eps) = 2^{20} \lambda^2 \log(1/\eps) + 2^{560} \lambda^{32}$. By Lemma~\ref{lem:easy:lower}, it follows that
$$|\Lambda \setminus \Lambda^*| \le \frac{\eps}{\lambda^3} \cdot \frac{n^2}{k} \binom{\lambda k/2}{k} \le \eps|\Lambda|,$$
as required.
\end{proof}

When $\lambda \in (2,3)$, the proof of Theorem~\ref{thm:structure:quant} implies the following weaker bound.

\begin{theorem}\label{thm:structure:small:lambdas}
For each $\gamma > 0$, there exists a constant $C(\gamma) > 0$ such that the following holds. Let $2 + \gamma \le \lambda \le 3$ and $\eps > 0$ be fixed, let $n$ be sufficiently large, and let $k \ge (\log n)^4$. If $A \subset [n]$ is chosen uniformly at random from those sets with $|A| = k$ and $|A+A| \le \lambda k$, then~there exists an arithmetic progression $P$ with
$$A \subset P \qquad \text{and} \qquad |P| \le \frac{\lambda k}{2} + C(\gamma) \log(1/\eps)$$ 
with probability at least $1 - 2\eps$.
\end{theorem}

Theorem~\ref{thm:structure:small:lambdas} follows by repeating the (entire) proof of Theorem~\ref{thm:structure:quant}, replacing (everywhere) the condition $\lambda \ge 3$ by the condition $\lambda \ge 2 + \gamma$, and the conditions $r(A) \ge c(\lambda,\eps)$ and $k \ge 2^{400} \lambda^{25}(\log n)^3$ 
by the conditions $r(A) \ge C(\gamma)$ and $k \ge (\log n)^4$. We leave the (straightforward, though somewhat tedious) details to the reader.

To finish the section, let us quickly deduce Corollary~\ref{thm:counting}.

\begin{proof}[Proof of Corollary~\ref{thm:counting}]
The lower bound follows from Lemma~\ref{lem:easy:lower} (see also Proposition~\ref{prop:lower:bound:counting}, below), so it remains to prove the upper bound. To do so, note that (by increasing the implicit constant in the upper bound if necessary) we may assume that $\log n \ge 2^{320} \lambda^{25}$, and hence we may apply Theorem~\ref{thm:structure:quant} with $\eps := 1/2$. Since there are at most $n^2 / k$ arithmetic progressions of length $\lambda k/2 + c(\lambda,\eps)$, it follows that 
$$|\Lambda| \le \frac{2n^2}{k} {\lambda k/2 + c(\lambda,\eps) \choose k} \le \exp\bigg( \frac{2c(\lambda,\eps)}{\lambda} \bigg) \frac{n^2}{k} {\lambda k / 2 \choose k} \le \exp\big( c(\lambda,\eps)\big) \cdot \frac{n^2}{k} {\lambda k / 2 \choose k},$$
as required.
\end{proof}

\section{The lower bounds}\label{lower:sec}

In this section, we prove lower bounds for the size of $\Lambda$, and for the typical size of the smallest arithmetic progression containing a set $A \in \Lambda$. The bounds we obtain indicate that the upper bounds in Theorem~\ref{thm:structure} and Corollary~\ref{thm:counting} are not far from best possible. We begin with the construction for the typical structure, which is very simple. 

\medskip
\pagebreak

\begin{prop}\label{prop:lower:bound:structure}
Given $\lambda \ge 4$, let $\eps > 0$ be sufficiently small, and let $n,k \in \N$ be sufficiently large. If $A \subset [n]$ is chosen uniformly at random from the sets with $|A| = k$ and $|A+A| \le \lambda k$, then with probability at least $\eps$,  
$$|P| \ge \frac{\lambda k}{2} + 2^{-6} \lambda^2 \log(1/\eps)$$
for every arithmetic progression $P$ containing $A$.
\end{prop}

\begin{proof}
Set $r := 2^{-6} \lambda^2 \log(1/\eps)$, and consider the family $\cA(r)$ of sets $A = A' \cup \{0,v\}$, where $A' \subset [\lambda k / 2 - 8r / \lambda ]$ with $|A'| = k - 2$, and $v = \lambda k / 2 + r$. We claim that most such sets satisfy $|A+A| \le \lambda k$. Indeed, since $A' + A' \subset [\lambda k - 16r / \lambda ]$, this holds as long as the set $\{x \in A' : x > \lambda k / 2 - r - 16r / \lambda \}$ has at most $16r / \lambda$ elements. If $k \ge 16r / \lambda$, then the expected number of elements of this set is  
$$\frac{k-2}{\lambda k / 2 - 8r / \lambda} \cdot \bigg( r + \frac{8r}{\lambda} \bigg) \,<\, \frac{2(\lambda + 8)}{\lambda - 1} \cdot \frac{r}{\lambda} \,\le\, \frac{8r}{\lambda},$$ 
since $\lambda \ge 4$, and it follows by Markov's inequality that $|A+A| \le \lambda k$ with probability at least $1/2$, as claimed. Now, observe that   
$$|\cA(r)| \,=\, {\lambda k / 2 - 8r/\lambda \choose k - 2} \,\ge\, \frac{2}{\lambda^2} \exp\bigg( - \frac{16r}{\lambda(\lambda - 1)} \bigg) {\lambda k /2 \choose k} \,\ge\, \frac{\sqrt{\eps}}{\lambda^2} {\lambda k /2 \choose k},$$
where the first inequality follows from the binomial inequalities
$${a \choose b - 2} \ge \frac{b^2}{2a^2} {a \choose b} \qquad \text{and} \qquad {a - c \choose b} \ge \bigg( 1 - \frac{b}{a - c} \bigg)^c {a \choose b},$$
again using the bound $k \ge 16r / \lambda$, and the second follows since $r \le 2^{-5} \lambda(\lambda-1) \log(1/\eps)$. Now, for each $a \in [n/\lambda k]$ and $b \in [n/4]$, and each set $A$ as above, we apply the linear map $x \mapsto ax + b$ to $A$. We obtain at least
\begin{equation}\label{eq:upper:structure:final:step}
\bigg\lfloor \frac{n}{\lambda k} \bigg\rfloor \cdot \frac{n}{4} \cdot \frac{1}{2} \cdot \frac{\sqrt{\eps}}{\lambda^2} {\lambda k /2 \choose k} \ge \, \eps^{2/3} \cdot \frac{n^2}{k} {\lambda k / 2 \choose k}
\end{equation}
distinct sets $A \subset [n]$ with $|A| = k$ and $|A+A| \le \lambda k$. 

Finally, note that few of these sets $A$ are contained in a shorter arithmetic progression, since such an arithmetic progression would have length at most $\lambda k / 4 + r/2 < \lambda k / 3$. Recalling the upper bound on $|\Lambda|$ given by Corollary~\ref{thm:counting}, and that $\eps$ was chosen sufficiently small, it follows that the right-hand side of~\eqref{eq:upper:structure:final:step} is at least $\eps |\Lambda|$, as required.
\end{proof}

Obtaining our lower bound on the size of $|\Lambda|$ will be slightly more delicate.

\medskip
\pagebreak

\begin{prop}\label{prop:lower:bound:counting}
If $\lambda \ge 2^{30}$ and $n,k \in \N$ are sufficiently large, then
\begin{equation}\label{eq:lower:bound:counting}
\big| \big\{ A \subset [n] : |A| = k, \, |A+A| \le \lambda k  \big\} \big| \ge \exp\big( 2^{-8} \lambda^{1/2} \big) \frac{n^2}{k} \binom{\lambda k/2}{k}.
\end{equation}
\end{prop}

In the proof of Lemma~\ref{prop:lower:bound:counting}, we will need the following simple bound on the number of independent sets of a given size in a graph.  

\begin{lemma}\label{cor:fkg}
Let $G$ be a graph with $n$ vertices, $m$ edges and $\ell$ loops. Let $R$ be a uniformly chosen random subset of $k$ vertices, where $k \leq \lfloor n/2 \rfloor$. If $\cB$ is the event that $R$ is an independent set, then
$$\Pr(\cB) \geq \exp\left( -\frac{9mk^2}{2n^2} - \frac{3\ell k}{n}\right) - \exp\left(-\frac{k}{16}\right).$$
\end{lemma}

Lemma~\ref{cor:fkg} is an almost immediate consequence of the FKG inequality for the hypergeometric distribution, see, e.g.,~\cite[Lemma~3.2]{BMSW}.

\begin{lemma}[Hypergeometric FKG Inequality]\label{lemma:HFKG}
Suppose that $\{B_i\}_{i \in I}$ is a family of subsets of an $n$-element set $\Omega$. Let $t \in \{0, \ldots, \lfloor n/2 \rfloor\}$, let $R$ be the uniformly chosen random $t$-subset of $\Omega$, and let $\cB$ denote the event that $B_i \nsubseteq R$ for all $i \in I$. Then for every $\eta \in (0,1)$,
$$\Pr(\cB) \ge \, \prod_{i \in I} \left(1 - \left(\frac{(1+\eta)t}{n}\right)^{|B_i|}\right) - \exp\big( - \eta^2t / 4 \big).$$
\end{lemma}

\begin{proof}[Proof of Lemma~\ref{cor:fkg}]
The claimed bound follows immediately from Lemma~\ref{lemma:HFKG}, applied with 
$t = k$ and $\eta = 1/2$, and with the sets $B_i$ being the edges and loops of $G$, using the fact that $1-x \geq e^{-2x}$ for $0 \le x \le 3/4$.
\end{proof}

\begin{proof}[Proof of Proposition~\ref{prop:lower:bound:counting}]
Set $c := 2^{-8}$ and $r := 2c \lambda^{3/2}$. We will first prove that there are at least $\exp\big( 2c \lambda^{1/2} \big) \binom{\lambda k/2}{k}$ subsets $A \subset [\lambda k/2 + r]$ of size $k$ with $|A+A| \le \lambda k$, each containing the endpoints $1$ and $\lambda k/2 + r$. Since this bound can be applied in each of the (at least) $n^2 / 4\lambda k$ arithmetic progressions of length $\lambda k/2 + r$ in $[n]$, and since the sets $A$ obtained for different arithmetic progressions are distinct, it will follow that 
$$|\Lambda| \ge \frac{n^2}{4\lambda k} \cdot \exp\big( 2c \lambda^{1/2} \big) \binom{\lambda k/2}{k} \ge \exp\big( c \lambda^{1/2} \big) \frac{n^2}{k} \binom{\lambda k/2}{k},$$
as required.

To prove the claimed bound, let $R$ be a uniformly chosen subset of $[2,\lambda k/2 + r-1]$ with exactly $k-2$ elements, and set $A := R \cup \{1, \lambda k/2 + r\}$. Observe first that
\begin{equation}\label{eq:lower:total:sets}
{\lambda k/2 + r - 2 \choose k - 2} \ge \frac{1}{\lambda^2} \bigg( \frac{\lambda k + 2r}{\lambda k} \bigg)^k {\lambda k/2 \choose k} \ge \exp\big( 3c \lambda^{1/2} \big) {\lambda k/2 \choose k},
\end{equation}
where the first inequality holds since ${a \choose b} = \frac{a(a-1)}{b(b-1)} {a - 2 \choose b - 2}$ and using~\eqref{eq:binomial:classic:too}, and the second follows
since $r = 2c \lambda^{3/2}$, and because $\lambda$ and $k$ were chosen sufficiently large. 

It will therefore suffice to prove that $|A+A| \le \lambda k$ with probability at least $\exp\big( - c \lambda^{1/2} \big)$. To do so, define
$$A' := \big\{ x \in A : x \le \lambda k / 2 - r \big\} \qquad \text{and} \qquad B := \big\{ x \in A : x > \lambda k / 2 - r \big\},$$
and set $b := 16c\lambda^{1/2}$. Observe that $\Ex[|B|] \le 4r / \lambda = b / 2$, and hence 
\begin{equation}\label{eq:lower:size:of:B}
\Pr\big( |B+B| \ge b^2 \big) \le \Pr\big( |B| \ge b \big) \le \exp\big( - c\lambda^{1/2} \big),
\end{equation}
by Hoeffding's inequality. We claim that, setting $X := [\lambda k - 2r + 1, \lambda k - 2r + b^2]$, we have
\begin{equation}\label{eq:lower:miss:big:set}
\Pr\big( (A'+B) \cap X = \emptyset \big) \ge 2 \cdot \exp\big( - c\lambda^{1/2} \big).
\end{equation}
Before proving~\eqref{eq:lower:miss:big:set}, observe that, together with~\eqref{eq:lower:total:sets} and~\eqref{eq:lower:size:of:B}, it will suffice to deduce the proposition. Indeed, if $(A'+B) \cap X = \emptyset$ and $|B+B| \le b^2 = |X|$, then 
$$|A+A| \le \lambda k - 2r + |(A'+B) \setminus [\lambda k - 2r] | + |B + B| \le \lambda k,$$
since $A'+ A' \subset [\lambda k - 2r]$ and $A' + B \subset [\lambda k]$, and noting that $b^2 = 2^8 c^2 \lambda \le 4 c \lambda^{3/2} = 2r$.

To prove~\eqref{eq:lower:miss:big:set} we will use Lemma~\ref{cor:fkg}. To do so, we define a graph $G$ with vertex set $[\lambda k/2 + r]$ and edge set
$$E(G) = \big\{ xy \,:\, x \le \lambda k/2 - r, \, y > \lambda k/2 - r \text{ and } x + y \in X \big\} \cup \big\{ x \,:\, x + \lambda k/2 + r \in X \big\}.$$ 
Observe that if $R$ is an independent set in $G$, then $(A'+B) \cap X = \emptyset$. Note that $G$ has at most $2r b^2 \le 2^{10} c^3 \lambda^{5/2}$ edges and at most $b^2 = 2^8 c^2 \lambda$ loops, and that 
$$\frac{9 \cdot 2^{10} c^3 \lambda^{5/2} k^2}{2(\lambda k / 2 + r)^2} + \frac{3 \cdot 2^8 c^2 \lambda k}{\lambda k / 2 + r} \le 2^{15} c^3  \lambda^{1/2} + 2^{11} c^2 \le c \lambda^{1/2} - 1,$$
since $c = 2^{-8}$ and $\lambda \ge 2^{30}$. It follows by Lemma~\ref{cor:fkg} that
$$\Pr\big( (A'+B) \cap X = \emptyset \big) \ge \exp\big( - c \lambda^{1/2} + 1 \big) - \exp \big( - k/16 \big) \ge 2 \cdot \exp\big( - c \lambda^{1/2} \big)$$
as required, since $k$ is sufficiently large. This completes the proof of Proposition~\ref{prop:lower:bound:counting}. 
\end{proof}

\appendix

\section{Proof of Theorem~\ref{thm:M:stability}}\label{app:stability}

We will prove Theorem~\ref{thm:M:stability} using Theorem~\ref{thm:M:containers}, and the following two lemmas from~\cite{C19}. 

\begin{lemma}[Corollary~3.3 of~\cite{C19}]\label{supersat}
Let $A,B\subset \Z$ be finite sets, and let $0 < \eps < 1/2$. If $|B|\geq (1/2 + \eps)|A|$, then there are at least $\eps^2 |B|^2$ pairs $(b_1,b_2)\in B^2$ such that $b_1+b_2 \not \in A$.
\end{lemma}

\begin{lemma}[Corollary~3.5 of~\cite{C19}]\label{satstability}
Let $m \in \N$ and $0 < \eps < 2^{-10}$. If $A,B \subset \mathbb{Z}$, with $|A| \le m$ and $(1-\eps)m / 2 \le |B| \le (1+2\eps)m / 2$, then one of the following holds:
\begin{itemize}
\item[$(a)$] There are at least $4 \eps^2 m^2$ pairs $(b_1,b_2)\in B^2$ such that $b_1+b_2\not\in A$.
\item[$(b)$] There is an arithmetic progression $P$ with $|P| \le m / 2 + 2^5 \eps m$, and a set $T \subset B$ with $|T| \le 8\eps m$ such that $B \setminus T \subset P$. 
\end{itemize}
\end{lemma}

Let us say that a set $B \subset [n]$ is $(\eps,m)$\textit{-close to an arithmetic progression} if there is an arithmetic progression $P$ with $|P| \le m / 2 + 2^5 \eps m$, and a set $T \subset B$ with $|T| \le 8 \eps m$ such that $B \setminus T \subset P$. Recall also from~\eqref{def:Lambda} the definition of $\Lambda$. 

\begin{proof}[Proof of Theorem~\ref{thm:M:stability}]
Set $Y := [n]$, $\eps := 4\gamma$ and $m := \lambda k \ge 2^{120} \lambda^2 (\log n)^3$, and let $\cA$ be the family of sets given by Theorem~\ref{thm:M:containers}. We prove the theorem via three simple claims.

\medskip
\noindent \textbf{Claim 1:} For every pair $(A,B) \in \cA$, either 
\begin{itemize}
\item[$(a)$] $|B| \le (1-\eps)\lambda k / 2$ or
\item[$(b)$] $|B| \le (1+2\eps)\lambda k / 2$, and $B$ is $(\eps,\lambda k)$-close to an arithmetic progression. 
\end{itemize}

\begin{proof}[Proof of Claim~1]
To see this, let $(A,B)\in \cA$ and suppose that $|B| \ge (1-\eps)\lambda k / 2$. By Theorem~\ref{thm:M:containers}$(ii)$, there are at most $\eps^2|B|^2$ pairs $b_1,b_2 \in B$ with $b_1+b_2\not\in A$. By Lemma~\ref{supersat}, it follows that $|B| \le (1+2\eps)\lambda k / 2$. Now, by Lemma~\ref{satstability}, and noting that $\eps^2|B|^2 < 4\eps^2\lambda^2 k^2$, it follows that there is an arithmetic progression $P$ with $|P| \le \lambda k / 2 + 2^5 \eps \lambda k$, and a set $T \subset B$ with $|T| \le 8\eps \lambda k$ such that $B \setminus T \subset P$, as required.
\end{proof}

Now, recall from Theorem~\ref{thm:M:containers}$(i)$ that for each set $J \in \Lambda$, there exists $(A,B) \in \cA$ such that $A \subset J+J$ and $J \subset B$. We first consider the pairs $(A,B) \in \cA$ with $|B| \le (1-\eps)\lambda k / 2$. 

\medskip
\noindent \textbf{Claim 2:} There are at most
$$e^{-\eps k/2} {\lambda k / 2 \choose k}$$
sets $J \in \Lambda$ such that $J \subset B$ for some $(A,B) \in \cA$ with $|B| \le (1-\eps)\lambda k / 2$.

\begin{proof}[Proof of Claim~2]
Recalling the bound~\eqref{eq:number:of:original:containers} on the size of $\cA$, it follows (using~\eqref{eq:binomial:classic:too}) that the number of sets $J$ is at most  
\begin{equation}\label{eq:claim2:counting}
|\cA| \cdot \binom{(1-\eps)\lambda k / 2}{k} \le \exp \Big( 2^{16}\eps^{-2}\sqrt{\lambda k}(\log n)^{3/2} - \eps  k \Big) \binom{\lambda k / 2}{k}.
\end{equation}
Now, recalling that $\eps \ge 2^{6} \lambda^{1/6} k^{-1/6} (\log n)^{1/2}$, it follows that the right-hand side is at most $e^{-\eps k / 2} {\lambda k / 2 \choose k}$, as claimed. 
\end{proof}

Finally, we consider the pairs $(A,B) \in \cA$ that satisfy property~$(b)$ of Claim~1. Let us write $\Lambda'$ for the family of sets $J \in \Lambda$ such that $J \setminus T$ is contained in an arithmetic progression of size $\lambda k / 2 + 2^{5} \eps \lambda k$ for some $T \subset J$ with $|T| \le 2^{7} \eps k$.

\medskip
\noindent \textbf{Claim 3:} There are at most
$$e^{-\eps k} {\lambda k / 2 \choose k}$$
sets $J \in \Lambda \setminus \Lambda'$ with $J \subset B$ for some $(A,B) \in \cA$ such that $|B| \le (1+2\eps)\lambda k / 2$, and $B$ is $(\eps,\lambda k)$-close to an arithmetic progression.

\begin{proof}[Proof of Claim~3]
Let $(A,B) \in \cA$, and suppose that $B \setminus T \subset P$ for some arithmetic progression $P$ and set $T \subset B$ with 
$$|P| \le \lambda k / 2 + 2^5 \eps \lambda k \qquad \text{and} \qquad |T| \le 8 \eps \lambda k.$$
Observe that there at most
\begin{equation}\label{eq:farAPfinal}
\sum_{s \ge 2^7 \eps k} \binom{(1+2\eps)\lambda k / 2}{k-s}\binom{8\eps \lambda k}{s}
\end{equation}
sets $J \in \Lambda \setminus \Lambda'$ with $J \subset B$. Indeed, $J \setminus T \subset P$, so if $|J \cap T| \le 2^7 \eps k$ then $J \in \Lambda'$. 

Note that the right-hand side of~\eqref{eq:farAPfinal} is zero if $\lambda < 2^4$, so we may assume that $\lambda \ge 2^4$. Now, observe that 
$$\binom{(1+2\eps)\lambda k / 2}{k-s} \binom{8\eps \lambda k}{s} \le (1 + 2\eps)^k \bigg( \frac{2}{\lambda - 2} \cdot \frac{8e\eps \lambda k}{s} \bigg)^s {\lambda k / 2 \choose k}.$$
Hence, summing~\eqref{eq:farAPfinal} over $(A,B)\in \cA$, and noting that $|\cA| \le e^{\eps k}$ (cf.~\eqref{eq:claim2:counting}), it follows that there are at most
$$e^{3\eps k} {\lambda k / 2 \choose k} \sum_{s \ge 2^7 \eps k} \bigg( \frac{2^6 \eps k}{s} \bigg)^s \le e^{-\eps k} {\lambda k / 2 \choose k},$$
as claimed.
\end{proof}

Now, recall that, by Theorem~\ref{thm:M:containers}$(i)$, for every $J \in \Lambda$ there exists $(A,B) \in \cA$ such that $A \subset J+J$ and $J \subset B$. Combining Claims~1,~2 and~3, it follows that
$$|\Lambda \setminus \Lambda'| \le \Big( e^{-\eps k / 2} + e^{-\eps k} \Big) {\lambda k / 2 \choose k} \le e^{-\gamma k} {\lambda k / 2 \choose k},$$
as required. 
\end{proof}

\section{Two binomial inequalities}\label{app:binomial}

In this section we will prove Observation~\ref{obs:choices:boring}. For the reader's convenience, we repeat the statement. Recall that $\alpha := 2^{25} \lambda^2 \delta = 2^{-7} \lambda^{-1}$. 

\begin{obs}
Let $b \le \delta k$ and $\mu > 2$, with $\lambda / 2 \le \mu \le 2\lambda$, and let $s \le t \le 2^{22} \lambda^2 b$. Then
$${\lambda k / 2 - \mu b - s \choose k - b - s} \le e^{\alpha b} \bigg( \frac{\lambda - 2}{\lambda} \bigg)^{\mu b} \bigg( \frac{2}{\lambda - 2} \bigg)^{b} {\lambda k / 2 - s \choose k - s},$$
and
$${\lambda k / 2 - \mu b/2 - s/ 2 - t + \delta b \choose k - b - s} \le e^{\alpha b} \bigg( \frac{\lambda - 2}{\lambda} \bigg)^{\mu b/2} \bigg( \frac{2}{\lambda - 2} \bigg)^{b} \binom{\lambda k / 2 - s/ 2 - t}{k - s}.$$
\end{obs}

\begin{proof}
To prove the first inequality, observe that, by~\eqref{fact:binomial_classic}, 
$${\lambda k / 2 - \mu b - s \choose k - b - s} \le \bigg( \frac{(\lambda - 2)k + 2b}{\lambda k - 2s} \bigg)^{\mu b} {\lambda k / 2 - s \choose k - b - s},$$
and that, by~\eqref{eq:binomial:classic:too},
$${\lambda k / 2 - s \choose k - b - s} \le \bigg( \frac{2k - 2s}{(\lambda - 2)k} \bigg)^{b} \binom{\lambda k / 2 - s}{k - s} \le \bigg( \frac{2}{\lambda - 2} \bigg)^{b} \binom{\lambda k / 2 - s}{k - s}.$$ 
Since $s \le 2^{22} \lambda^2 b$ and $\mu \le 2\lambda$, it follows that
$${\lambda k / 2 - \mu b - s \choose k - b - s} \le \exp\bigg( \frac{2^{25} \lambda^2 b^2}{k} \bigg) \bigg( \frac{\lambda - 2}{\lambda} \bigg)^{\mu b} \bigg( \frac{2}{\lambda - 2} \bigg)^{b} \binom{\lambda k / 2 - s}{k - s},$$
as claimed. For the second inequality, observe first that $s \le t \le 2^{22} \lambda^2 b \le k/8$, and that $\mu b / 2 \le \lambda b \le k/8$. Since $\lambda \ge 3$, it follows, using~\eqref{eq:another:binomial:inequality} 
(with $c = \delta b$), that 
$${\lambda k / 2 - \mu b/2 - s/ 2 - t + \delta b \choose k - b - s} \le e^{4\delta b} {\lambda k / 2 - \mu b/2 - s/ 2 - t  \choose k - b - s}.$$
We now repeat the proof of the first part: by~\eqref{fact:binomial_classic}, we have
$${\lambda k / 2 - \mu b/2 - s/2  - t \choose k - b - s} \le \bigg( \frac{(\lambda - 2)k + 2b + s - 2t}{\lambda k - s - 2t} \bigg)^{\mu b} {\lambda k / 2 - s/2 - t \choose k - b - s}.$$
and, by~\eqref{eq:binomial:classic:too},
$${\lambda k / 2 - s/2 - t \choose k - b - s} \le \bigg( \frac{2k - 2s}{(\lambda - 2)k + s - 2t} \bigg)^{b} \binom{\lambda k / 2 - s/2 - t}{k - s}.$$ 
Since $s \le t \le 2^{22} \lambda^2 b$ and $\mu \le 2\lambda$, it follows that
$${\lambda k / 2 - \mu b/2 - s/ 2 - t + \delta b \choose k - b - s} \le \exp\bigg( \frac{2^{25} \lambda^2 b^2}{k} \bigg) \bigg( \frac{\lambda - 2}{\lambda} \bigg)^{\mu b} \bigg( \frac{2}{\lambda - 2} \bigg)^{b} \binom{\lambda k / 2 - s/2 - t}{k - s},$$
as claimed.
\end{proof}

\section{The proof in the case $\lambda k \ge n$}\label{app:boring}

In this section we will deal with some minor technical issues that arise when $\lambda k \approx 2n$, and hence complete the proof of Theorem~\ref{thm:structure:quant}. First, we need the following variant of Lemma~\ref{lem:red-to-int}. 

\begin{lemma}\label{lem:red-to-int:app}
Let $\lambda \ge 3$ and $n,k \in \N$, with $n \le \lambda k \le 2n$. Then
$$|\mathcal{F} \setminus \Lambda^*| \le \big( n - \lambda k / 2 \big) \cdot |\cI|.$$
\end{lemma}

\begin{proof}
We repeat the proof of Lemma~\ref{lem:red-to-int}, except we now set $d := 1$. To be precise, let $A \in \cF \setminus \Lambda^*$, choose $a \in \N$ minimal such that the sets
$$\big\{ x \in A : x \le a \big\} \qquad \text{and} \qquad \big\{ x \in A : x > a + \lambda k / 2 \big\}$$
are both non-empty and together contain at most $\delta k$ elements, define $\varphi(A) := A - a$, and observe that $\varphi(A) \in \cI$ (cf.~the proof of Lemma~\ref{lem:red-to-int}). Now, for each set $S \in \cI$, there are at most $n - \lambda k / 2$ choices for $a$ such that $a + S \subset [n]$, and therefore 
$$|\varphi^{-1}(S)| \le n - \lambda k / 2$$
for every $S \in \cI$, as required.
\end{proof}

We also need the following variant of Lemma~\ref{lem:easy:lower}.

\begin{lemma}\label{lem:easy:lower:app}
Let $\lambda \ge 3$ and $n,k \in \N$, with $n \le \lambda k \le 2n$. Then
$$\big| \big\{ A \subset [n] : |A| = k, \, |A+A| \le \lambda k  \big\} \big| \ge \frac{1}{\lambda^2} \cdot \big( n - \lambda k / 2 + 1 \big) \binom{\lambda k/2}{k}.$$
\end{lemma}

\begin{proof}
We consider, for each arithmetic progression $P$ of length $\lambda k / 2$ in $[n]$, all subsets $A \subset P$ of size $k$ containing both endpoints of $P$. All of these sets are distinct, and all satisfy $|A+A| \le \lambda k$. There are $n - \lambda k / 2 + 1$ choices for the arithmetic progression, and therefore
$$|\Lambda| \ge \big( n - \lambda k / 2 + 1 \big) {\lambda k / 2 - 2 \choose k - 2} \ge \frac{1}{\lambda^2} \cdot \big( n - \lambda k / 2 + 1 \big) \binom{\lambda k/2}{k},$$
as claimed.
\end{proof}

We can now deduce Theorem~\ref{thm:structure:quant} in the case $\lambda k \ge n$.  

\begin{proof}[Proof of Theorem~\ref{thm:structure:quant}]
Observe that the theorem is trivial if $\lambda k / 2 \ge n$ (since in this case $P = [n]$ satisfies the conditions), so let us assume that  $n \le \lambda k < 2n$. Replacing Lemma~\ref{lem:red-to-int} by Lemma~\ref{lem:red-to-int:app} in the proof of Lemma~\ref{lem:stability}, and recalling that $\eps > e^{-\delta^2 k}$, we obtain
$$|\Lambda \setminus \Lambda^*| \le \big( n - \lambda k / 2 \big) \cdot |\cI| + \frac{\eps}{2\lambda^2} {\lambda k / 2 \choose k}.$$
Now, by Lemmas~\ref{lem:sparse:quant} and~\ref{lem:dense:quant}, we have 
$$|\cI| = |\cS| + |\cD| \le \frac{\eps}{2\lambda^2} {\lambda k/2 \choose k}.$$
Finally, by Lemma~\ref{lem:easy:lower:app}, it follows that
$$|\Lambda \setminus \Lambda^*| \le \frac{\eps}{\lambda^2} \cdot \big( n - \lambda k / 2 + 1 \big) \binom{\lambda k/2}{k} \le \eps|\Lambda|,$$
as required.
\end{proof}

\section*{Acknowledgements}

We are grateful to the anonymous referee for a very careful reading of the proof, and for numerous helpful suggestions which improved the presentation.

\end{document}